\newif\ifspringer
\newif\ifelsevier
\elseviertrue 

\ifspringer
\RequirePackage{fix-cm} 
\documentclass[smallcondensed,nospthms]{svjour3}
\usepackage[paperwidth=15.51cm,paperheight=23.51cm,left=1.75cm,right=1.75cm,top=2cm,bottom=2cm]{geometry} 
\fi

\ifelsevier
\documentclass[preprint,3p,times]{elsarticle}
\fi

\usepackage[english]{babel}
\usepackage[utf8]{inputenc}
\usepackage{pdflscape} 
\usepackage{amssymb}
\usepackage{amsmath}
\usepackage{amsthm}
\usepackage{mathtools}
\ifspringer
\usepackage{bm} 
\fi
\usepackage{enumitem}
\usepackage[normalem]{ulem}
\usepackage{marvosym}
\usepackage[dvipsnames,svgnames]{xcolor}
\usepackage{array}
\usepackage{multirow}

\usepackage{tikz-cd}
\tikzcdset{scale cd/.style={every label/.append style={scale=#1},
    cells={nodes={scale=#1}}}}
\usetikzlibrary{matrix,decorations.pathreplacing, calc, positioning,fit}

\usepackage{graphicx}
\ifspringer
\usepackage[FIGBOTCAP,TABTOPCAP,scriptsize]{subfigure} 
\fi
\ifelsevier
\usepackage[FIGBOTCAP,TABTOPCAP]{subfigure} 
\fi
\usepackage{tikz}
\usetikzlibrary{calc,shapes,arrows,backgrounds}
\tikzstyle{every picture}+=[remember picture]
\tikzset{fontscale/.style = {font=\relsize{#1}}}
\usepackage{pgfplots} 

\usepackage{epstopdf} 
\usepackage[ruled,vlined,linesnumbered]{algorithm2e}
\usepackage{setspace}

\usepackage{adjustbox}

\usepackage{hyperref}
\hypersetup{colorlinks=true,linkcolor=blue,citecolor=blue,filecolor=blue,urlcolor=blue}
\usepackage[all]{hypcap}

\mathtoolsset{showonlyrefs} 
\allowdisplaybreaks 
\ifspringer
\smartqed 
\fi
\ifelsevier
\biboptions{sort&compress,numbers}
\fi

\usepackage{xcolor,cancel}

\usepackage{fancyvrb}

\usepackage{xspace}
\usepackage{framed}

\makeatletter
\DeclareRobustCommand\onedot{\futurelet\@let@token\@onedot}
\newcommand{\@onedot}{\ifx\@let@token.\else.\null\fi\xspace}
\newcommand{\eg}{{e.g}\onedot} 
\newcommand{\ie}{{i.e}\onedot}

\newcommand{\st}{{s.t}\onedot}

\makeatother

\newtheorem{prop}{Proposition}
\newtheorem{rem}{Remark}

\theoremstyle{definition}
\newtheorem{defn}{Definition}
\newtheorem{exmp}{Example}

\newcommand{\vectbi}[1]{\boldsymbol{#1}} 
\newcommand{\vect}[1]{\vectbi{#1}}

\renewcommand{\leq}{\leqslant}
\renewcommand{\geq}{\geqslant}

\newcommand{\KKK}{\mathcal{K}}
\newcommand{\PPP}{\mathcal{P}}

\newcommand{\XXX}{\mathcal{X}}
\newcommand{\SSS}{\mathcal{S}}

\ifspringer
\DeclareMathAlphabet{\mathcalb}{OMS}{cmsy}{b}{n} 
\DeclareMathAlphabet{\mathcal}{OMS}{cmsy}{m}{n} 

\fi
\ifelsevier

\newcommand{\bd}{\mathbf{d}}
\newcommand{\bs}{\vect{s}}
\newcommand{\bt}{\vect{t}}
\newcommand{\bA}{\mathrm{A}}
\newcommand{\bM}{\mathrm{M}}
\newcommand{\bN}{\mathrm{\bf{N}}}
\newcommand{\bNz}{\mathrm{\bf{N0}}}
\newcommand{\bIN}{\mathrm{\bf{IN}}}
\newcommand{\bINz}{\mathrm{\bf{IN0}}}
\newcommand{\bG}{\mathrm{\bf{G}}}
\newcommand{\bDelta}{\vect{\Delta}}

\fi

\newlength{\casesvsep}
\newcommand{\figpath}{figure}
\graphicspath{{\figpath/}}

\makeatletter
\newcommand*{\shifttext}[2]{%
\settowidth{\@tempdima}{#2}%
\makebox[\@tempdima]{\hspace*{#1}#2}%
}
\makeatother

\newcommand{\qtext}[1]{``#1''}

\ifspringer

\fi
\ifelsevier

\fi

\ifspringer
\journalname{\dots}
\fi

\makeatletter
\providecommand{\doi}[1]{%
  \begingroup
    \let\bibinfo\@secondoftwo
    \urlstyle{rm}%
    \href{http://dx.doi.org/#1}{%
      doi:\discretionary{}{}{}%
      \nolinkurl{#1}%
    }%
  \endgroup
}
\makeatother

\newtheorem{thm}{Proposition}

  \makeatletter
\def\ps@pprintTitle{%
   \let\@oddhead\@empty
   \let\@evenhead\@empty
   \def\@oddfoot{\reset@font\hfil\thepage\hfil}
   \let\@evenfoot\@oddfoot
}
\makeatother

\begin{document}
\newcommand{\titletext}{
Stable numerical evaluation of multi-degree B-splines
}
\newcommand{\titlerunningtext}{
Stable numerical evaluation of multi-degree B-splines
}

\newcommand{\abstracttext}{
Multi-degree splines are piecewise polynomial functions having sections of different degrees. 
They offer significant advantages over the classical uniform-degree framework, as they allow for modeling complex geometries with fewer degrees of freedom and, at the same time, for a more efficient engineering analysis. 
Moreover they possess a set of basis functions with similar properties to standard B-splines.
In this paper we develop an algorithm for efficient evaluation of multi-degree B-splines, which, unlike previous approaches, is numerically stable. 
The proposed method consists in explicitly constructing a mapping between a known basis and the multi degree B-spline basis of the space of interest, exploiting the fact that the two bases are related by a sequence of knot insertion and/or degree elevation steps and performing only numerically stable operations.
In addition to theoretically justifying the stability of the algorithm, we will illustrate its performance through numerical experiments that will serve us to demonstrate its excellent behavior in comparison with existing methods, which, in some cases, suffer from apparent numerical problems.
}

\ifspringer
\newcommand{\separ}{\and}
\fi
\ifelsevier
\newcommand{\separ}{\sep}
\fi

\newcommand{\keytext}{
Multi-degree spline \separ B-spline basis \separ matrix representation \separ stable evaluation \separ algorithmic computation \separ Greville absciss\ae
}

\newcommand{\MSCtext}{
65D07 \separ 65D15 \separ 41A15 \separ 68W40
}

\ifspringer
\title{\titletext
}
\titlerunning{\titlerunningtext} 

\author{
Carolina Vittoria Beccari \and
Giulio Casciola 
}
\authorrunning{C.V.~Beccari, G.~Casciola} 

\date{Received: date / Accepted: date}

\maketitle

\begin{abstract}
\abstracttext
\subclass{\MSCtext}
\end{abstract}
\fi

\ifelsevier
\begin{frontmatter}

\title{\titletext}

\author[label1]{Carolina Vittoria Beccari\corref{cor1}}
\ead{carolina.beccari2@unibo.it}
\author[label1]{Giulio Casciola}
\ead{giulio.casciola@unibo.it}

\cortext[cor1]{Corresponding author.}

\address[label1]{Department of Mathematics, University of Bologna,
Piazza di Porta San Donato 5, 40126 Bologna, Italy}

\begin{abstract}
\abstracttext
\end{abstract}


\end{frontmatter}
\fi

\section{Introduction}
\label{sec:intro}
Spline functions are the foundation of numerous results and methods of approximation theory and, nowadays, are an integral part of geometric modeling and computational engineering analysis systems.
Classically, a univariate spline is a piecewise function defined on a partition of a real interval $[a, b]$, where each piece belongs to the space of algebraic polynomials of degree less than or equal to $d\geq0$ and where two pieces are joined with continuity at most $C^{d-1}$.
The success of splines is largely due to the fact that they possess a \emph{B-spline basis}, namely a normalized, totally positive basis of compactly supported functions \cite{dB78,Schu2007}.
Besides the elegance of the theoretical framework, this basis has excellent properties from the computational point of view, both for its good conditioning, and because its evaluation can be carried out  via an efficient and numerically stable algorithm, the well-known Cox-de Boor recurrence scheme \cite{Cox1972,dB1972}. These classical splines will be hereinafter referred to also as \emph{conventional splines}. 

As the name suggests, \emph{multi-degree splines} (\emph{MD-splines}, for short) are a generalization of conventional splines where each piece can have a different degree.
They are a natural and extremely powerful extension of the classical framework, which allows for modeling complex geometries with fewer control points and at the same time leads to more efficient engineering analysis \cite{TSHH2017}.
Although the concept of multi-degree splines is long-standing \cite{NSSS84,SZS2003}, for several years the interest in these spaces has been mostly theoretical. Only recently, in fact, has it been understood how to build a set of functions analogue to the B-spline basis, dubbed \emph{MDB-spline basis} (or simply \emph{MDB-splines}). This has opened up the possibility of easily integrating multi-degree splines into current computing systems and has made them a real full-fledged extension of conventional splines. Multivariate versions of the multi-degree concept have been as well devised \cite{LIU201642,coreform2018,TSHH2017}.

The first approaches for constructing an MDB-spline basis can be traced back to the work by Shen and Wang \cite{SW2010b,SW2010a} and are subject to constraints on the continuity attained at the joins. More importantly, they rely on integral recurrence relations, which, as firstly observed in \cite{SW2013}, are widely recognized to be overly complicated and of little practical use.

Subsequently, alternative and more computationally practicable methods were proposed.
Ideally, one would want evaluation techniques based on algebraic recurrence relations, in the spirit of the famous Cox-de Boor's method. However, such recurrence schemes have been identified and proven to exist only for particular multi-degree spline spaces and precisely those where pieces of different degree are joined with continuity at most $C^1$ \cite{BC2019,Li2012}.

Methods capable of dealing with multi-degree spaces with arbitrary structure stem from a common basic idea, which is to map a set of known (or easily computable) functions into the basis of interest. They can be traced back to two different approaches.
The first consists in determining the basis functions by interpolation, exploiting the fact that, under suitable assumptions, Hermite interpolation problems are unisolvent in MD-spline spaces \cite{BMuhl2003}. This involves solving a number of (small) linear systems that represent the continuity conditions at the joins.  
Following this approach, in \cite{BCM2017} normalized MDB-splines are expressed as combinations of \emph{transition functions} (a notion earlier introduced in the context of local spline interpolation \cite{ABC2013a, BCR2013a}), which allows to efficiently compute  their expansion with respect to the collection of the Bernstein bases relative to the breakpoint intervals.
In \cite{BCR2017} the same idea was used to deal with splines whose pieces are drawn from Extended Chebyshev spaces \cite{Maz2011b,Schu2007}, a powerful and versatile extension of algebraic polynomials.

The second approach stems from the idea of calculating in an explicit way, namely without having to solve any linear system, a matrix operator $\bM$ that specifies the mapping between a known basis (or collection of bases) and the set of MDB-splines. The resulting matrix representation $\bN = \bM \bN_0$ provides a way to evaluate the MDB-spline basis functions, which are the elements of vector $\bN$, as a combination of easily computable functions, which are the elements of vector $\bN_0$.
This avenue was firstly pursued in \cite{TSHH2017}, which underpinning idea is to gather the continuity constraints between spline pieces in a matrix and then calculate its null space by a recursive procedure. 
The follow-up paper \cite{TSHMH2018} proves that the output of the algorithm is exactly the entire set of MDB-splines, whereas implementation details are given in  \cite{SPE2018} and a Chebyshevian extension of the construction is presented in \cite{chebMD2019}. 
In these series of papers, the vector $\bN_0$ is composed of a collection of local bases, which can be, in particular, either the Bernstein bases relative to the breakpoint intervals or conventional B-spline bases, each one relative to a sequence of intervals of equal degree. In both cases the functions in $\bN_0$ are discontinuous.

In \cite{BC2020} it is observed that the functions in $\bN$ and $\bN_0$ are related through a sequence of successive knot insertions (the same observation was made in the context of Chebyshevian splines in \cite{chebMD2019})
and that matrix $\bM$ can be computed by inverting these steps, giving rise to a process called \emph{reverse knot insertion (RKI)}. 
This realization allowed the authors of \cite{BC2020} not only to generate the matrix representation \cite{TSHH2017} in a more direct and intuitive way, but more generally to derive a matrix representation with respect to any set of functions $\bN_0$ which are related to the basis $\bN$ through the aforementioned knot-insertion structure. 
It is shown, in particular, that to minimize the number of performed operations, it is convenient to start from a basis $\bN_0$ composed of 
conventional B-spline functions connected with $C^0$ continuity.
Such functions are the MDB-spline basis of a piecewise conventional spline space, referred to as a \emph{$C^0$ MDB-spline} space, and, as such can be evaluated by known techniques.
The same paper also deals with how to generate the matrix representation when $\bN_0$ is the conventional B-spline basis having maximum (over all intervals) degree and same continuities as the basis $\bN$. In this case the procedure consists in inverting the sequence of local degree elevations connecting the target space and the maximum-degree conventional space
and is therefore called \emph{reverse degree elevation (RDE)}. 
It is also discussed how it is possible to combine successive reverse knot insertion and reverse degree elevation steps in a unique algorithm. This algorithm hence allows to construct a matrix representation in the most general case, that is under the sole assumption that the space spanned by $\bN_0$ contains the space spanned by $\bN$.

However, both methods \cite{TSHH2017} and \cite{BC2020} have a weakness, which is that they require to calculate higher-order derivatives of B-spline (or Bernstein) basis functions (the order of the derivatives to be calculated  corresponds to the maximum continuity or maximum degree to be handled).
The evaluation of these derivatives can be carried out in a stable way \cite{BUT1976}. However, not only is it a price to pay in terms of computational cost, but also, and above all, it can lead to the numerical instability of the algorithm using them. In fact, B-spline derivatives may easily become very large numbers for high differentiation order and/or very nonuniform partitions, hence the arithmetic operations carried out with them are potentially risky. 
The actual occurrence of instability phenomena was observed in \cite{BC2020}, where it is suggested to work with a \emph{compensated} version of the algorithm in order to improve on accuracy
(compensation is a standard technique, see, e.g., \cite{Higham2002}). 

The main contribution of this paper is a new, stable algorithm that provides the matrix representation of any MDB-spline basis. The algorithm exploits a suitable reformulation of the reverse knot insertion and reverse degree elevation processes, thanks to which only numerically stable operations are performed and, in particular, no derivative needs to be evaluated.
We will show how a natural way to arrive at this reformulation  is to pass through the concept of Greville absciss\ae. The Greville absciss\ae\ are defined, similarly as in the case of conventional splines, as the coefficients of the identity function in the MDB-spline basis. As is well known, they are essential in various applications ranging from approximation and interpolation to isogeometric analysis. We will show that these absciss\ae\ can be obtained by integrating the MDB-spline basis of the corresponding derivative space.

All in all, the resulting stable algorithm has the form of a triangular scheme. As such it has quadratic computational complexity (vs.\ the linear complexity of previous algorithms), nevertheless the computation time is negligible in practical situations.
An appealing feature of the approach lies in its deep relationship with the usual spline tools of knot insertion and degree elevation,  which use automatically ensures the correctness of the set of MDB-splines provided as output. This is an important difference with respect to the approach in \cite{TSHH2017}, where it is required to prove a posteriori that the generated functions are MDB-splines.

For ease of presentation and in the interest of clarity, we will develop in all details the so-called \emph{RKI Algorithm}, corresponding to the case where $\bN$ and $\bN_0$ are related by iterated reverse knot insertions. 
The circumstance in which the two basis vectors are related by degree elevation (giving rise to the \emph{RDE Algorithm}) can be addressed by similar general principles and will be discussed more briefly in the last part of the paper. 
Finally, we will illustrate how it is possible to mix RDE and RKI steps, like in \cite{BC2020}, in such a way as to be able to choose the initial basis vector  $\bN_0$ that will entail the least number of operations and thereby improve the efficiency of the computation. The latter algorithm builds on the previous two and due to space constraints we will limit ourselves to providing a quick sketch of the procedure. \\

The remainder of the paper is organized a follows. Section
\ref{sec:preliminaries} collects the necessary notions and results on
multi-degree splines and their matrix representation. Section \ref{sec:RKI_rev} presents the new algorithm, in particular showing how to compute the Greville absciss\ae, using these absciss\ae\ to reformulate the reverse knot insertion process without resorting to MDB-spline derivatives and finally introducing the triangular scheme.
Section \ref{sec:general_algo} is concerned with the computational/inplementation aspects of the procedure and presents a practical example of its application. 
The numerical stability of the method is discussed theoretically in  section
\ref{sec:stability}, while subsection \ref{sec:num_cons_RKI} proposes a series of numerical experiments which, in addition to confirming the theoretical predictions, highlight the potential inaccuracy of previous methods.
Section \ref{sec:RDE} illustrates how to derive a matrix
representation in terms of the conventional B-spline basis of maximum degree.
Conclusions are drawn in section \ref{sec:conclusion}. 

\section{Background and basic notions}
\label{sec:preliminaries}

In this section we gather the notions and results on multi-degree splines of interest for this paper. 

\subsection{Multi degree (MD) spline spaces and B-spline bases}
Throughout the paper we will deal with piecewise functions, with pieces drawn from polynomial spaces whose dimensions are allowed to change from interval to interval. Spaces of such functions are defined as follows.

\begin{defn}[Multi degree spline space] \label{def:MD-space}
Let $[a,b]$ be a closed bounded real interval, $\mathcal{X}=\{x_i\}_{i=1}^{q}$ be
  a partition of $[a,b]$ s.t.\ $a\eqqcolon x_0<x_1<\ldots <x_q<x_{q+1}\coloneqq b$ and $\bd=(d_0,\dots,d_q)$ be a vector of nonnegative integers. Let also $\KKK=(k_1,\dots,k_q)$ be a vector of nonnegative integers such that $k_i \leq\min\{d_{i-1},d_{i}\}$. 
The corresponding space of \emph{multi-degree splines (MD-splines, for short)} is the set of functions
 \[
\begin{aligned}
\SSS(\PPP_{\bd},\XXX,\KKK) \coloneqq & \left\{ f \,\big|\, \mbox{there exist } p_i\in\PPP_{d_i}, i=0,\dots,q,  \mbox{ such that:} \right. \\
&
\begin{minipage}[b]{0.75\linewidth}
\begin{enumerate}[label=\roman*)]
 \item $f(x)=p_i(x)$ for $x\in [x_i,x_{i+1}], \, i=0,\dots,q$;\\[-1.5ex]
 \item $D^\ell p_{i-1}(x_i)=D^\ell p_{i}(x_i)$ for $\ell=0,\dots,k_i, \, i=1,\dots,q \left. \vphantom{\big|} \right\},$
\end{enumerate}
\end{minipage}
\end{aligned}
\]
where $\PPP_{d}$ is the space of algebraic polynomials of degree at most $d$.
\end{defn}

Note that the above definition returns a conventional spline space in the particular case where $\bd$ is a vector with constant entries, that is $d_0=\dots=d_q$, and therefore conventional splines can be seen as a subclass of MD-splines.

A space $\SSS\coloneqq \SSS(\PPP_{\bd},\KKK,\bDelta)$ has dimension
$
K \coloneqq \dim(\SSS) =d_0+1+\sum_{i=1}^q(d_i-k_i). 
$

Moreover, as shown in \cite{BCM2017}, it possesses 
a B-spline-type basis, dubbed \emph{MDB-spline basis} (or \emph{MDB-splines}), sharing many properties with conventional B-splines. 
For defining this basis we shall introduce two partitions $\bs$ and $\bt$ as follows: 
\begin{equation}\label{eq:s}
\bs \coloneqq\{s_j\}_{j=1}^K\coloneqq  \{\underbrace{a,\dots,a}_{d_0+1 \text{ times}}, \underbrace{x_1,\dots,x_1}_{d_1-k_1 \text{ times}}, \dots, \underbrace{x_q,\dots,x_q}_{d_q-k_q \text{ times}} \} ,
\end{equation}
and
\begin{equation}\label{eq:t}
\bt \coloneqq\{t_j\}_{j=1}^K\coloneqq \{ \underbrace{x_1,\dots,x_1}_{d_0-k_1 \text{ times}}, \dots, \underbrace{x_q,\dots,x_q}_{d_{q-1}-k_q \text{ times}}, \underbrace{b,\dots,b}_{d_{q}+1 \text{ times}} \}.
\end{equation}
We call $\bs$ and $\bt$ the \emph{left} and \emph{right extended partition}, respectively, associated with the MD-spline space.

Denoted $m\coloneqq\max_i \{d_i\}$, the MDB-spline basis functions $N_{1,m},\dots,N_{K,m}$ are defined recursively over $\bs$ and $\bt$. The recurrence consists in constructing, for $n=0,\ldots,m$, a sequence of functions $N_{i,n}$, $i=m+1-n,\ldots,K$, where $N_{i,n}$ is supported on $[s_i,t_{i-m+n}]$ and is determined on each nontrivial interval 
$[x_j,x_{j+1}) \subset [s_i,t_{i-m+n}]$ 
by the following integral relation \cite{BCM2017}: 
\begin{equation}\label{def:int_rec}
\displaystyle
N_{i,n}(x)\coloneqq \left \{
\begin{array}{ll}
1, \qquad x_j \leq x < x_{j+1} & n=m-d_j, \vspace{0.2cm}\\
\int_{-\infty}^x\left[\delta_{i,n-1}N_{i,n-1}(u)-\delta_{i+1,n-1}N_{i+1,n-1}(u)\right]du,  & n>m-d_j, \vspace{0.2cm}\\
0,  & otherwise,
\end{array}
\right.
\end{equation}
where
\begin{equation}\label{def:int_rec2}
\delta_{i,n}\coloneqq\left( \int_{-\infty}^{+\infty} N_{i,n}(x) dx \right)^{-1}.
\end{equation}

In \eqref{def:int_rec} we assume undefined $N_{i,n}$ functions to be zero and, in this case, we set

\begin{equation} \label{cond_N=0}
\int_{-\infty}^x\delta_{i,n}N_{i,n}(u)du\coloneqq \left\{
\begin{array}{ll}
0, & x< s_i,\\
1, & x \geq s_i.\\
\end{array}
\right.
\end{equation}

The $K$ functions generated by the above integral formulation 
possess analogous characterizing properties as conventional B-splines, that is: 
\begin{enumerate}[label=\roman*)]
\setlength{\itemsep}{4pt}
\item \label{propty:BS1} \emph{Compact support}: $N_{i,m}(x)=0$ for $x \notin [s_i, t_i]$;
\item \label{propty:BS2} \emph{Positivity}:  $N_{i,m}(x) > 0$ for $x \in (s_i,t_i)$;
\item \label{propty:BS3} \emph{End point property}: $N_{i,m}$ vanishes exactly $d_{ps_i}-\max \{ j\geq 0 \ | \ s_i=s_{i+j} \} \;$ times at $s_i$ and
  $d_{pt_i-1}-\max \{ j \geq 0 \ | \ t_{i-j}=t_i \} \;$ times at $t_i$, 
where $ps_i$ and $pt_i$ are \st $x_{ps_i}=s_i$ and $x_{pt_i}=t_i$;
\item \label{propty:BS4} \emph{Partition of unity}: $\displaystyle \sum_i N_{i,m}(x) = 1$, $\forall x \in [a,b]$.
\end{enumerate}

Moreover, the above properties \ref{propty:BS1}, \ref{propty:BS3} and \ref{propty:BS4} also warrant the uniqueness of the MDB-spline basis (see \cite{BCM2017} for further details). We refer to previous papers for illustrations of MDB-spline basis functions, see \eg \cite{BCM2017,BC2020,TSHH2017,SPE2018}. 

\subsection{$C^0$ Multi-degree splines}\label{sec:C0}
In the remainder of the paper we will often rely on MD-spline spaces whose elements are conventional spline functions connected with $C^0$ continuity, which we refer to as \emph{$C^0$ MD-splines}.
 
These spaces are convenient tools to work with, in that 
well-established methods can easily be adapted to deal with them.  
In particular a generalization of Cox de-Boor recurrence
formula \cite{Cox1972,dB1972}, the main method for evaluating conventional B-splines, 
is given in \cite[Proposition 4]{BC2020} 
along with a recurrence relation for the computation of derivatives (\cite[Proposition 5]{BC2020}). 

The integrals of $C^0$ MDB-splines can as well be efficiently evaluated resorting to existing results. 
To this aim, it suffices to observe that a $C^0$ MD-spline $N_{i,m}$ can be of only one of the following two types: 
either it is a conventional B-spline, which means that within its support each piece has same degree, or it consists of only two non-trivial pieces of different degrees connected with $C^0$ continuity. 
In both cases, recalling that the integral of a conventional degree-$d$ B-spline basis function $N_{i,d}$ is equal to the ratio between the  support width of $N_{i,d}$ and $d+1$, its integral is easily computed as:
\begin{equation} \label{eq:intC0}
\int_{x_{ps_i}}^{x_{pt_i}}
  N_{i,m} \; dx = \sum_{j=ps_i}^{pt_i-1} \frac{x_{j+1}-x_j}{d_j+1},
\end{equation}
with $ps_i$ and $pt_i$ defined as in \ref{propty:BS3}.

\begin{defn}[Associated $C^0$ MD-spline space]\label{def:asC0}
The  $C^0$ MD-spline space  associated with a
multi-degree spline space $\SSS(\PPP_{\bd},\XXX,\KKK)$ as in Definition \ref{def:MD-space} is the unique
MD-spline space $\SSS(\PPP_{\bd},\XXX,\KKK_0)$ having same breakpoint sequence
  and degree vector as $\SSS(\PPP_{\bd},\XXX,\KKK)$ and vector of continuities
  $\KKK_0=(k^0_1,\dots,k^0_q)$ such that $k_i^0=0$ if $d_{i-1}\neq d_i$ and $k_i^0=k_i$ otherwise.
\end{defn}

\subsection{Matrix representation}
If we take an MD-spline space (including, possibly, a conventional spline space) $\SSS_0$ such that 
$\SSS\subset \SSS_0$, we can write the relationship between the respective MDB-spline bases in the form 
\begin{equation} \label{eq:M}
\bN = \bM \, \bN_0, 
\end{equation}
where  $\bN \coloneqq \left(N_1,\dots,N_{K}\right)$ and
$\bN_{0}=\left(N_{1}^0,\dots,N_{K_0}^0\right)$ are the vectors
containing the basis functions of $\SSS$ and $\SSS_0$, respectively, and
$\bM$ is a linear operator of size $K\times K_0$.
We refer to the above as the \emph{matrix representation of $\bN$ relative to $\bN_0$}.  
Knowing $\bM$ and $\bN_0$, one can thus use \eqref{eq:M} to evaluate $\bN$. 

One way to compute matrix $\bM$ is to choose $\SSS_0$ to be the $C^0$ MD-spline space associated with $\SSS$
and rely on iterated application of the following result, referred to as \emph{Reverse Knot Insertion}, RKI for short \cite[Proposition 6]{BC2020}.

\begin{prop}[Reverse knot insertion]\label{prop:RKI}
Let $\SSS\equiv\SSS(\PPP_{\bd},\XXX,\KKK)$ and
  $\widehat{\SSS}\equiv\SSS(\PPP_{\bd},\XXX,\widehat{\KKK})$ be MD-spline spaces with same breakpoint sequence and degrees. Let also the respective continuity vectors be $\KKK=(k_1,\dots,k_j,\dots,k_{q})$ and $\widehat{\KKK}=(k_1,\dots,k_j-1,\dots,k_q)$, for  $j\in \{1,\dots,q\}$. 
Then, $\SSS\subset\widehat{\SSS}$ and the corresponding MDB-spline bases $\{N_{i,m}\}_{i=1}^K$ and $\{\hat{N}_{i,m}\}_{i=1}^{K+1}$
are related through
\begin{equation} \label{eq:kiN}
N_{i,m} = \alpha_{i} \hat{N}_{i,m} + (1-\alpha_{i+1}) \hat{N}_{i+1,m}, \quad i=1,\ldots,K,
\end{equation}
where, being $\ell$ the index of the element of $\bs$ such that  $s_{\ell}\leq x_j < \min(s_{\ell+1},b)$, the coefficients $\alpha_i$ are such that 
\begin{equation} \label{eq:alpha_ki} 
\alpha_i \left\{\begin{array}{ll}
  =1, & i=1, \ldots, \ell-d_j, 
  \\[1ex]
\in \  ]0,1[\ , & i=\ell-d_j+1, \ldots, \ell-d_j+k_j, \\[1ex]
=0, & i=\ell-d_j+k_j+1, \ldots, K+1.
\end{array} \right.
\end{equation}
Moreover, the coefficients $\alpha_i$, $i=\ell-d_j+1, \ldots,\ell-d_j+{k_j}$, can be computed from the MDB-splines $\{\hat{N}_{i,m}\}_{i=1}^{K+1}$ 
by the relation
\begin{equation} \label{eq:alpha2} 
  \alpha_{i} = 1+\alpha_{i-1} \frac{D_{-}^{k_j} \hat N_{i-1,m}|_{x_j}
    - D_{+}^{k_j} \hat N_{i-1,m}|_{x_j}}
    {D_{-}^{k_j} \hat N_{i,m}|_{x_j} - D_{+}^{k_j} \hat N_{i,m}|_{x_j}}.
  \end{equation}
\end{prop}

\begin{rem}\label{rem:KIvsRKI}
Knot insertion is a well-established tool for conventional splines and was generalized to multi-degree splines in 
\cite{BCM2017}. 
Classically, using knot insertion we pass from the representation in a space $\SSS$ to that in a space $\widehat \SSS $ in which the continuity at a breakpoint is decreased. The word \emph{reverse} refers to the fact that we go from space $\widehat \SSS$ to $\SSS$ by increasing the continuity at a breakpoint.
Moreover, under the assumptions of Proposition \ref{prop:RKI}, classical knot insertion would mean to compute the 
coefficients \eqref{eq:alpha_ki} knowing the MDB-spline bases of both spaces $\SSS$ and $\widehat{\SSS}$.
Conversely, in reverse knot insertion the coefficients \eqref{eq:alpha2} only depend on the basis 
$\{\hat N_{i,m}\}$ of $\widehat \SSS$ and can hence be used to compute the basis $\{N_{i,m}\}$ of $\SSS$.
\end{rem}

To derive the matrix representation \eqref{eq:M}, take a sequence of MD-spline spaces, all defined on $[a,b]$, sharing same breakpoint sequence $\XXX$ and degree vector $\bd$, such that
\begin{equation} \label{eq:nested}
\SSS\coloneqq \SSS_g \subset \dots \subset \SSS_1 \subset \SSS_0.
\end{equation}
Furthermore suppose that each space $\SSS_r$, $r=1,\dots,g$, is obtained from 
$\SSS_{r-1}$ increasing by one the continuity at a breakpoint in such a way that $K_r\coloneqq\dim(\SSS_r) =\dim(\SSS_{r-1})-1$.
Relation \eqref{eq:kiN} to pass from $\SSS_{r-1}$ to $\SSS_{r}$ can be written
in the matrix form $\bN_r = \bA_r \, \bN_{r-1}$, with a bidiagonal matrix $\bA_r$ of size $K_r\times (K_r+1)$ having on each row the coefficients $\alpha_i^r$  and 
$(1-\alpha_{i+1}^r)$ relating the bases $\bN_r$ and $\bN_{r-1}$ as in \eqref{eq:kiN}, \eqref{eq:alpha_ki}.
Iterating the procedure for as many multiplicities and knots as needed, one
obtains the matrix M in \eqref{eq:M}, which is the product of all  matrices $\bA_r$, that is $\bM=\bA_g\cdots\bA_1$. 

Being based on repeated reverse knot insertions, the above procedure for the
construction of the representation matrix in \eqref{eq:M} is called  \emph{RKI Algorithm} \cite{BC2020}. 
In order for the RKI Algorithm -- that is the matrix representation it produces -- to be an efficient tool for evaluating the target MDB-spline basis $\bN$, the vector $\bN_0$ should be known
(or more precisely computable through established methods).
This is the reason why we have chosen it to contain the basis functions of the $C^0$ MD-spline space associated with 
$\SSS$. The described procedure, however, simply requires that $\SSS$ can be
generated from $\SSS_0$ by repeated reverse knot insertions. It is therefore
possible, and sometimes desired, to choose $\SSS_0$ in a different way, see
Remark \ref{rem:H} for further details.

It shall be noted that, according to equation \eqref{eq:alpha2}, the described 
method requires calculating the derivatives of MDB-spline
basis functions at each breakpoint $x_j$ to be processed, up to the target continuity $k_j$. 
As pointed out in \cite{BC2020}, 
B-spline derivatives, and MDB-spline derivatives likewise, may become very large numbers when the order of differentitation increases, causing potential numerical issues for very large degrees and/or nonuniform partitions.
To circumvent this criticality, in that paper it is proposed to work with a \emph{compensated} version of the algorithm \cite{Higham2002}.
Although this strategy is able to improve the accuracy of the results, 
the presence of high order derivatives remains an intrinsic feature of the existing algorithms which would be far preferable to avoid.

\begin{rem}\label{rem:H}
Our assumption that $\SSS_0$ be a  $C^0$ MD-spline space is merely dictated by simplicity and conciseness of presentation.
However, with a view to choosing $\SSS_0$ in such a way that its basis can be evaluated by known methods, one can as well consider alternative initial spaces, since the general algorithm proposed in \cite{BC2020} is based on the sole requirement that $\SSS_0$ be defined on $[a,b]$, have same breakpoint sequence as $\SSS$ and $\SSS\subset \SSS_0$.
In particular we may want to take $\SSS_0$ to be a \emph{piecewise} space and its basis vector $\bN_0$ to be the collection of the Bernstein bases relative to the breakpoint intervals, or alternatively a piecewise conventional B-spline basis on a sequence of abutting intervals with equal degree (in the latter situation, the generated matrix $\bM$ is the H-operator proposed in \cite{TSHH2017}). The ideas presented in this work can easily be adapted to both of these situations.

Lastly, we may want $\SSS_0$ to be a conventional spline space of
  degree $m\coloneqq\max_i\{d_i\}$, whose basis $\bN_0$ will be a conventional
  degree-$m$ B-spline basis. This situation cannot be addressed by reverse knot
  insertions, but by a similar approach, as well described in \cite{BC2020},
  based on \emph{reverse degree elevation} and will briefly be discussed in
  section \ref{sec:RDE}.

\end{rem}

\section{The novel RKI algorithm}
\label{sec:RKI_rev}
In the following we present the general ideas our new algorithm is based on. 
To formalize our method we will need to use spaces spanned by derivatives of MD-splines, that are defined as follows.

\begin{defn}[Spline space of derivatives]
\label{def:der_space}
We denote by $D^{r}\SSS\coloneqq\{D_+^{r}f \;  | \; \allowbreak f\in \SSS\}$
the function space whose elements are $r$th right derivatives of functions in a multi degree spline space $\SSS$. There follows that $D^{r}\SSS$ has dimension $K-r$, where $K$ is the dimension of $\SSS$. 
\end{defn}

Throughout the paper, for brevity, we simply write $D\SSS$ in place of $D^{1}\SSS$.
We will be concerned with spaces $D^{r}\SSS$, with $r$
ranging from $0$ up to $\max_i \{d_i\}$, the latter corresponding to the number of levels in the recurrence \eqref{def:int_rec}. 
Therefore $D^{r}\SSS$ may contain functions discontinuous at breakpoints (the right derivative in the definition accounting for possible discontinuities) and, rather than a \qtext{single} MD-spline space, it should be regarded as \qtext{piecewise} space, whose functions are MD-splines defined on abutting intervals and possibly
discontinuous at breakpoints. 
In particular, we associate with $D^{r}\SSS$ the vectors of degrees
$(d_0-r,\dots,d_q-r)$ and continuities $(k_1-r,\dots,k_q-r)$, which may be
negative integers, in such a way that $K^{(r)}\coloneqq \dim(D^{r}\SSS)=K-r=d_0-r+1+\sum_{i=1}^q(d_i-k_i)$, with the convention that  the
restriction of $D^{r}\SSS$ to $[x_i,x_{i+1}]$ be the zero function in case
$d_i-r<0$. 
We can as well define a piecewise MDB-spline basis $N_1^{(r)},\dots,\allowbreak
N_{K-r}^{(r)}$ spanning $D^r\SSS$, which will be relative to the left and right
extended partitions $\bs^{(r)}$ and $\bt^{(r)}$ obtained by replacing $d_i$ and
$k_i$ with $d_i-r$ and $k_i-r$ in \eqref{eq:s}, \eqref{eq:t}. These correspond to the
functions generated by the the integral recurrence relation \eqref{def:int_rec} for $n=m-r$ and defined on the  partitions $\bs$ and $\bt$ relative to $\SSS$ in such a way that:  
\begin{equation} \label{eq:basisid}
N_1^{(r)}(x)=N_{r+1,m-r}(x),\quad \dots \quad ,N_{K-r}^{(r)}(x)=N_{K,m-r}(x), \quad \forall x \in [a,b].
\end{equation}

Furthermore, we can generalize to MD-spline spaces the classical notion of \emph{Greville absciss\ae}, that are 
the coefficients of the expansion of the function $f(x)=x$ in the B-spline basis of a conventional spline space containing linear functions.
For a multi-degree spline space $D^r\SSS$ the Greville absciss\ae\, relative to the basis 
$\{ N_{i,m}^{(r)} \}$ are defined to be the coefficients $\xi_i^{(r)}$,
$i=1,\dots,K^{(r)}$, such that $\sum_{i=1}^{K^{(r)}} \xi_i^{(r)}
N_{i,m}^{(r)}(x)=x$, $\forall x\in [x_j,x_{j+1}]$,  $j=0,\ldots,q$, such that $d_j\geq1$.

The following proposition relates the Greville absciss\ae\ to the MDB-spline basis of the derivative space $D\SSS$.
A similar result was proved in \cite[Theorem 15]{CaMaPe2016} for Chebyshevian splines 
with all section spaces of the same dimension. 

\begin{prop}[Computation of Greville abscissae]
\label{prop:Greville}
Let $\SSS$ be a $K$-dimensio- nal MD-spline space. Then the Greville abscissa\ae \ $\xi_1,\dots,\xi_K$ with respect to the MDB-spline basis $N_1\dots,\allowbreak N_K$ 
of $\SSS$ are given by:
\begin{equation} \label{eq:Greville}
\xi_i=a+\sum_{j=1}^{i-1} \int_a^b N_j^{(1)}(x) dx,  \quad i=1,\ldots,K,
\end{equation}
where $N_j^{(1)}$, $j=1,\dots,K-1$, is the MDB-spline basis of the derivative space $D\SSS$. 
\end{prop}

\begin{proof}
The properties of the MDB-spline basis $\{N_j\}$ entail that $\xi_1=a$ and $\xi_K=b$. Moreover, from the partition of unity property and 
Abel's lemma we obtain: 
\[
\begin{aligned}
\sum_{i=1}^K \xi_i N_{i,m}(x) &= \xi_1\sum_{i=1}^K N_{i,m}(x)+\sum_{i=1}^{K-1} (\xi_{i+1}-\xi_i)\sum_{\ell=i+1}^K  N_{\ell,m}(x)\\&=a+\sum_{i=1}^{K-1} (\xi_{i+1}-\xi_i)\sum_{\ell=i+1}^K  N_{\ell,m}(x).
\end{aligned}
\]
Using relations \eqref{def:int_rec2} with $n=m-1$, \eqref{eq:basisid} with $r=1$ and \eqref{eq:Greville}:
\[
\delta_{i+1,m-1}^{-1}=\int_{-\infty}^{+\infty} N_{i+1,m-1}(u)du=\int_a^b N_{i}^{(1)}(u)du=\xi_{i+1}-\xi_i,
\]
By the recurrence definition \eqref{def:int_rec}:

\[
\sum_{\ell=i+1}^K  N_{\ell,m}(x)=\int_{-\infty}^x \delta_{i+1,m-1}N_{i+1,m-1}(u) du=\delta_{i+1,m-1}\int_{a}^x N_{i}^{(1)}(u) du.
\]
From the partition of unity property of the functions $N_{i}^{(1)}(x)$, $i=1,\dots,K-1$ , we thus obtain 
\[
\sum_{i=1}^K \xi_i N_{i,m}(x) = a+\sum_{i=1}^{K-1} \int_{a}^x N_{i}^{(1)}(u) du=a+\int_a^x du=x,
\]
which implies that $\xi_1,\dots,\xi_K$ are the Greville absciss\ae\ with respect to the basis $N_1\dots,N_K$ and concludes the proof.
Note that \eqref{eq:Greville} guarantees that the Greville absciss\ae\, form an increasing sequence. 
\end{proof}

The following result provides an alternative way to calculate the coefficients of reverse knot insertion, which, unlike previous methods \cite[Proposition 6]{BC2020}, 
 does not involve differentiating the MDB-spline basis.
\begin{prop}[Reverse knot insertion by Greville absciss\ae]
\label{prop:RKI_greville}
Under the same setting and assumptions of Proposition \ref{prop:RKI}, denoted by $\xi_i$ and $\hat \xi_i$ the Greville abscissae of $\SSS$ and
  $\widehat{\SSS}$, respectively,  the coefficients $\alpha_i$ in \eqref{eq:alpha_ki} can be computed as follows: 
\begin{equation} \label{eq:KI_rev} 
  \alpha_{i} = \frac{\hat \xi_i - \xi_{i-1}}{\xi_i - \xi_{i-1}}, \quad i=\ell-d_j+1, \ldots,
\ell-d_j+{k_j}.
  \end{equation}
\end{prop}

\begin{proof}
Let $f$ be a function in $\SSS\subset \widehat{\SSS}$ with expansions in the MDB-spline bases of $\SSS$ and $\widehat{\SSS}$:
\[
f(x) = \sum_{i=1}^Kc_iN_{i,m}(x)=\sum_{i=1}^{K+1}\hat{c}_i\hat{N}_{i,m}(x).
\]
According to \eqref{eq:alpha_ki} the above coefficients $c_i$ and $\hat c_i$ satisfy the relationship:  
\begin{equation} \label{eq:ki_spl_coeff}
\hat{c}_i =
\begin{cases}
c_i, & i \leq \ell-d_j, \quad \\
\alpha_{i}\,c_i + (1-\alpha_{i})\,c_{i-1}, & \ell-d_j+1 \leq i \leq \ell-d_j+k_j,\\
c_{i-1}, & i \geq \ell-d_j+k_j+1,
\end{cases}
\end{equation}\
with $s_{\ell}\leq x_j < \min(s_{\ell+1},b)$. 
By taking $f(x)=x$, equation \eqref{eq:KI_rev}  is hence obtained
from the middle line of \eqref{eq:ki_spl_coeff}.
\end{proof}

By virtue of the above results, the reverse knot insertion step leading from 
$\widehat{\SSS}$ to $\SSS$ can be outlined by a triangular scheme of the form

\vspace{-0.2cm}
\begin{equation} \label{eq:basic_block}
\end{equation}
\vspace{0.1cm}
\vspace{-1.6cm}
\begin{center}
\begin{tikzcd}[column sep={3em}]
D\SSS 
\arrow{rd}[inner sep=1pt]{G} 
 &  \\
\hspace{-0.2cm}\widehat{\SSS}    
\arrow{r}[pos=0.25,inner sep=2pt]{\;\;\;\; RKI}  
& \SSS
\end{tikzcd}
\end{center}

\noindent
The \qtext{G} arrow and the \qtext{RKI} arrow are both needed to evaluate $\SSS$. More precisely, to evaluate $\SSS$ we need information about two spaces. One is the space $\widehat{\SSS}$, obtained from $\SSS$ by inserting a knot, the other is the space $D\SSS$, spanned by the derivatives of functions in $\SSS$.
Concerning $D\SSS$ we need the MDB-spline basis in order to compute the Greville absciss\ae\ of $\SSS$ (Proposition \ref{prop:Greville}), whereas about $\widehat{\SSS}$ we need the Greville absciss\ae\ in order to compute the RKI coefficients 
(Proposition \ref{prop:RKI_greville}).
Hence we need to put ourselves in a position where the necessary quantities relative to $\widehat{\SSS}$ and $D\SSS$ are known.
To this end, the main idea is to concatenate several blocks of type \eqref{eq:basic_block} giving rise to a triangular scheme such as the one in \eqref{eq:tri}, as will be detailed in the following. This triangular scheme will be constructed in such a way that the information at the starting level (the first column of the scheme) is known (or easy to calculate). The remaining elements in the triangle can be derived by recursive application of the basis block \eqref{eq:basic_block}, moving from left to right and progressively generating the information relative to each column, up to reaching the vertex of the triangle, which finally contains the information relative to the target space.    

\begin{rem}As already noted, the reverse knot-insertion formula \eqref{eq:alpha2} entails computing higher order derivatives of the B-spline basis functions. On the contrary, the above triangular scheme does not involve the calculation of derivatives of any order.
In fact, one might be misled by the fact that the B-spline basis of $ D\SSS $ appears in the formula for the Greville abscissa, however, this basis can be evaluated directly (i.e. without involving any differentiation), as we shall see shortly.   
\end{rem}

\begin{defn}[$C^r$ join of two multi-degree spline spaces]
\label{def:join_md}
Let $\SSS_L$ and  $\SSS_R$ be MD-spline spaces on adjacent intervals, $[a,b]$ and $[b,c]$ ($a<b<c$) respectively. We define the \emph{$C^r$ join of $\SSS_L$ and $\SSS_R$}
 to be the MD-spline space on $[a,c]$ whose restriction to $[a,b]$ and $[b,c]$ coincides with $\SSS_L$ and $\SSS_R$, respectively, and whose elements are $C^r$ continuous at $b$. 
\end{defn}

It shall be noted that, according to the definition above, the $C^0$ join of two conventional spline spaces is a $C^0$ MD-spline space, whereas the $C^0$ join of two MD-spline spaces is not, in general, a $C^0$ MD-spline space. 

One can join with $C^r$ continuity two multi-degree spline spaces $\SSS_L$ and $\SSS_R$ defined on abutting intervals $[a,b]$ and $[b,c]$
by performing iterated reverse knot insertions at point $b$, each of which increases the continuity at the join. 
More precisely, denoted by  $\SSS$ and $\SSS_0$, respectively, the $C^r$ and $C^0$ join of the two spaces 
$\SSS_L$ and $\SSS_R$,
we shall start from $\SSS_0$ and, by RKI, generate the $C^1$ join of $\SSS_L$ and $\SSS_R$, which we denote by $\SSS_1$. 
We shall then iterate the procedure generating a sequence of nested spaces
$\SSS_k$ as in \eqref{eq:nested}, where $\SSS_{k}$ is the $C^{k}$ join of $\SSS_L$ and $\SSS_R$, $\SSS_{k+1}\subset \SSS_k$ and $\dim(\SSS_{k+1})=\dim(\SSS_{k})-1$, $k=1,\dots,r-1$.  In this way, the last space $\SSS_r$ will be the target space $\SSS$.
These joins will be performed by concatenating several basic blocks of type \eqref{eq:basic_block}. To this end,  we will need to use the MDB-spline bases $\bNz_{L}$ and $\bNz_{R}$ of the $C^0$ MD-spline spaces associated with $\SSS_L$ and $\SSS_R$ and the representation matrices $\bM_L$ and $\bM_R$ such that 
\[
\bN_L=\bM_L \bNz_L\quad \text{and} \quad \bN_R=\bM_R \bNz_R,
\] 
where $\bN_{L}$ and $\bN_{R}$ are the MDB-spline bases of  $\SSS_L$ and $\SSS_R$, respectively.
Moreover,  on account of the previous discussion, we will as well need such information for all the derivative spaces $D^{r-n}\SSS_L$ and $D^{r-n}\SSS_R$, for $n=0,\ldots,r$, as detailed in the following.

Being $K_L$ the dimension of $\SSS_L$, $\bG_L=(\xi_{L,1},\dots,\xi_{L,K_L})$ the vector of its Greville absciss\ae \  and $\bN_L=(N_{L,1},\dots,N_{L,K_L})$ the vector of the MDB-spline basis functions, with a similar notation for $\SSS_R$,
the $C^0$ join of $\SSS_L$ and $\SSS_R$, which we indicate by $[\SSS_L,\SSS_R]$, is an MD-spline space of dimension $K_L+K_R-1$ having 
 MDB-spline basis 
\begin{equation} \label{eq:Njoin}
\bN=[\bN_{L},\bN_{R}]\coloneqq \left(N_{L,1},\dots, N_{L,K_L-1}, N_{L,K_L}+N_{R,1},N_{R,2},\dots, N_{R,K_R}  \right),
\end{equation}
and Greville absciss\ae\ 
\[
\bG = [\bG_L,\bG_R]\coloneqq \left(\xi_{L,1}, \dots, \xi_{L,K_L}\equiv \xi_{R,1},  \xi_{R,2}, \dots,\xi_{R,K_R}\right).
\]

\noindent
Furthermore, the matrix representation of $\bN$ relative to the basis $\bNz=[\bNz_{L},\bNz_R]$ of the associated $C^0$ MDB-spline space is:
\begin{equation} \label{eq:join}
  \bN = [\bM_L,\bM_R] \bNz, \quad \text{with} \quad [\bM_L,\bM_R]\coloneqq
\begin{tikzpicture}[
  baseline,
  label distance=10pt 
]
\matrix [matrix of math nodes,left delimiter=(,right delimiter=),row sep=0.05cm,column sep=0.05cm] (m) {
         {} &  {}     &    {}  &  {} &  {}   \\
       {} & \bM_L  &    {}  &  {} &  {}   \\
    {} &  {}    & \ast   & {}&  {}    \\
       {}&  {}& {} &\bM_R &  {} \\
       {}&  {}   &  {}    &    {}   & {}\\
      };

\draw[dashed] ($0.5*(m-1-3)+0.5*(m-1-4)$) -- ($0.5*(m-3-3)+0.5*(m-4-4)$)--($0.5*(m-3-1.north)+0.5*(m-4-1)$) -- ($0.5*(m-1-1)+0.5*(m-1-1)$)--($0.5*(m-1-3)+0.5*(m-1-4)$) ;

\draw[dashed]  ($0.5*(m-2-2)+0.5*(m-3-3)$) -- 
 ($0.5*(m-2-5.north)+0.5*(m-3-5)$)--
  ($0.5*(m-5-5)+0.5*(m-5-5)$)  -- ($0.5*(m-5-2)+0.5*(m-5-3)$) --($0.5*(m-2-2)+0.5*(m-3-3)$);
\end{tikzpicture} 
\end{equation}
where the above entry $\ast$ has the same value in $\bM$, $\bM_L$ and $\bM_R$.
Note that, if $\SSS_L$ is a conventional spline space, then  $\bM_L$ is trivially known, being the identity matrix of size $K_L$, and similarly for $\bM_R$. 

We can regard the operation $[\, , ]$ as the \emph{$C^0$ join} of the representation matrices, of the vectors of MDB-spline basis functions or of those of Greville absciss\ae. 
Note that this join operation acts differently depending on the entity to which it is applied.

To describe the process of concatenating several blocks of type \eqref{eq:basic_block} we will need to indicate the aforementioned spaces $\SSS_k$,  by 
a double index, that is $\SSS_k \eqqcolon \SSS^{r,k}$, $k=1,\dots,r$. Suppose
for example that $\SSS_L$ and $\SSS_R$ are to be joined with $C^3$ continuity
at $b$ to generate space $\SSS$. Iterated application of \eqref{eq:basic_block} will lead to  the following triangular scheme of \qtext{size} 4: 

\vspace{3cm}
\begin{equation} \label{eq:tri}
\end{equation}

\vspace{-4cm}
\begin{center}
\begin{tikzcd}[column sep={3em}]
  {[D^3\SSS_L,D^3\SSS_R]=\SSS^{0,0}}\coloneqq D\SSS^{1,1} 
  \arrow{rd}[inner sep=1pt]{G} 
 & & & \\
  {[D^2\SSS_L,D^2\SSS_R]=\SSS^{1,0}}\coloneqq D\SSS^{2,1}    
\arrow{rd}[inner sep=1pt]{G}  
\arrow{r}[pos=0.25,inner sep=2pt]{\;\;\;\; RKI}  
& 
\SSS^{1,1}\coloneqq D\SSS^{2,2} 
\arrow{rd}[inner sep=1pt]{G} 
& & \\ 
  {[D\SSS_L,D\SSS_R]=\SSS^{2,0}}\coloneqq D\SSS^{3,1} 
\arrow{rd}[inner sep=1pt]{G} 
\arrow{r}[pos=0.25,inner sep=2pt]{\;\;\;\; RKI}  &
\SSS^{2,1}\coloneqq D\SSS^{3,2} 
\arrow{rd}[inner sep=1pt]{G} 
\arrow{r}[pos=0.25,inner sep=2pt]{\;\;\;\; RKI}  & 
\SSS^{2,2}\coloneqq D\SSS^{3,3} 
\arrow{rd}[inner sep=1pt]{G} &
\\
  {\SSS_0=[\SSS_L,\SSS_R]=}\SSS^{3,0}
\arrow{r}[pos=0.25,inner sep=2pt]{\;\;\;\; RKI}  &
\SSS^{3,1}
\arrow{r}[pos=0.25,inner sep=2pt]{\;\;\;\; RKI}  & 
\SSS^{3,2}
\arrow{r}[pos=0.25,inner sep=2pt]{\;\;\;\; RKI}   &
\SSS^{3,3}= \SSS
\end{tikzcd}
\end{center}

\vspace{0.2cm}
More generally, the $C^r$ join of $\SSS_L$ and $\SSS_R$ will involve a similar
triangular scheme having size $r+1$, where the element in row $n$ and column
$k$ is $\SSS^{n,k}$, for $n=0,\dots,r$, $k=0,\dots,n$. 
To implement the scheme all we need to know are the Greville abscissae and
MDB-spline bases of all spaces  $S^{n,0}$,
$n=0,\dots,r$ (that is all spaces in the leftmost column), since the information on all other columns and rows of the triangle
can be derived by progressive application of the basic block
\eqref{eq:basic_block}. Recalling that $\SSS^{n,0}$ is the $C^0$ join of $D^{r-n}\SSS_L$ and $D^{r-n}\SSS_R$, these Greville absciss\ae\, and MDB-spline bases are easily obtained from the described $C^0$ join operation $[\, , ]$.

We shall now use the above triangular scheme to compute the matrix
representation \eqref{eq:M} where $\bN$ is the MDB-spline basis of $\SSS$ and $\bNz$ the MDB-spline basis of the associated $C^0$ MD-spline
space. To this aim, observe that all spaces
$\SSS^{n,0},\dots,\SSS^{n,n}$ on a row of the triangle are
associated with the same $C^0$ MD-spline space according to Definition \ref{def:asC0},
which basis will be indicated by $\bNz^n$. We may then rewrite the triangular scheme
replacing $\SSS^{n,k}$ with its MDB-spline basis $\bN^{n,k}$ and Greville
absciss\ae\ $\bG^{n,k}$ and defining $\bM^{n,k}$ to be the matrix such that
$\bN^{n,k}=\bM^{n,k}\bNz^n$. In particular, matrices $\bM^{n,0}$ are obtained by
joining the representation matrices of $D^{r-n}\SSS_L$ and $D^{r-n}\SSS_R$ as in \eqref{eq:join}, whereas, for each $k>0$, 
matrix $\bM^{n,k}$ is obtained from the preceding one by the relation
$\bM^{n,k}=\bA^{n,k}\bM^{n,k-1}$, where  $\bA^{n,k}$  is the the bidiagonal
matrix whose nontrivial entries are the RKI coefficients to pass from
$\SSS^{n,k-1}$ to $\SSS^{n,k}$.

Relying on the matrix representation it is also easy to compute the vectors
$\bIN^{n,k}$
containing integrals of functions in $\bN^{n,k}$ that are needed for calculating the Greville absciss\ae. 
In particular, let $\bN_L^{n}$ and $\bN_R^{n}$ be the vectors containing the MDB-spline basis functions in 
$D^{r-n}\SSS_L$ and $D^{r-n}\SSS_R$, respectively,  and $\bIN_L^{n}$ and $\bIN_R^{n}$ be the vectors of their integrals. 
Then the vectors $\bIN^{n,0}$, corresponding to spaces 
$\SSS^{n,0}$ in the first column of the triangular scheme, 
contain the integrals of functions in
$[\bN_L^{n},\bN_R^{n}]$, $n=0,\dots,r$, and we will indicate this by
$\bIN^{n,0}=[\bIN_L^{n},\bIN_R^{n}]$.
For all spaces appearing in the subsequent columns, instead, the integrals of
basis functions are derived from the relation $\bN^{n,k}=\bM^{n,k} \bNz^n$,
which yields $\bIN^{n,k}=\bM^{n,k} \bINz^n$.

The described procedure for computing the $C^r$ join of spaces $\SSS_L$ and
$\SSS_R$ is outlined in Algorithm \ref{alg:1}.
The algorithm takes as input the Greville absiss\ae, the integrals and the
representation matrices of the MDB-spline bases of the two spaces to be joined and of their derivative spaces up to suitable order.
On account of their ease of computation, the integrals of the $C^0$ MDB-spline
bases $\bNz^n$ (Algorithm \ref{alg:1}, line 2) are evaluated at runtime using \eqref{eq:intC0}, but they could as well be provided as input.
The algorithm returns as output the representation matrix $\bM^{r,r}$, relative to the $C^0$
MD-spline space associated with the join space $\SSS$. Moreover, in anticipation of having to further join
the generated MD-spline space, it computes and returns all the matrices $\bM^{n,n}$,
$n=0, \ldots, r-1$, related to the derivative spaces
$\SSS^{n,n}=D^{r-n}\SSS$. 
Note that the overall procedure does never use the MDB-spline bases of the initial
spaces $\SSS^{n,0}$, for $n = 0,\ldots,r-1$, but just the integrals
$\bIN_L^{n}$, $\bIN_R^{n}$ and $\bINz^n$. \\

\begin{algorithm}
 \label{alg:1}
   \caption{Matrix representation of the $C^{r}$ join of two MD-spline spaces $\SSS_L$ and $\SSS_R$}
   {\small
\setstretch{1}
\SetAlgoLined
\KwData{$\bG_L^{n}$, $\bINz_L^n$, $\bM_L^n$, $\bG_R^{n}$, $\bINz_R^n$, $\bM_R^n$, $n=0,\dots,r$.}
  \KwResult{$\bM^{n,n}$, $n=0,\dots,r$.}
    \For{$n \leftarrow  0$ \KwTo $r$}{
      $\bINz^{n} \leftarrow [\bINz_L^n,\bINz_R^n]$  \;
      $\bM^{n,0} \leftarrow [\bM_L^{n},\bM_R^{n}]$  \;
      $\bG^{n,0} \leftarrow [\bG_L^{n},\bG_R^{n}]$  \;
    }
\For{$n \leftarrow 1$ \KwTo $r$}{
   \For{$k \leftarrow 1$ \KwTo $n$}{
   $\bIN^{n-1,k-1} \leftarrow \bM^{n-1,k-1} \cdot \bINz^{n-1}$\;
   Compute $\bG^{n,k}$ 
   by $\bIN^{n-1,k-1}$ using \eqref{eq:Greville} \;
   Compute the RKI coefficients from $\bG^{n,k}$ and $\bG^{n,k-1}$ using \eqref{eq:KI_rev}
     \;
    $\bM^{n,k} \leftarrow \bA^{n,k} \bM^{n,k-1}$\;
   }
}
}
\end{algorithm}

At this point, by repeatedly joining MD-spline spaces on abutting intervals, we can generate the matrix representation of an MDB-spline basis vector $\bN$, spanning an arbitrary space $\SSS \equiv \SSS (\PPP_ {\mathbf {d}}, \XXX, \KKK) $, with respect to the basis vector $\bNz$ of the associated 
$C^0$ MD-spline space $\SSS_0(\PPP_{\mathbf{d}}, \XXX, \KKK_0)$.
Essentially what we need to do is \qtext{break} the target space into a sequence of conventional spline spaces and join these spaces in pairs with the required continuities.

In particular, with reference to Definition \ref{def:asC0},
let $\text{J}$ be the vector containing the indices, in ascending order, of the breakpoints separating intervals with different degrees, including the first and last breakpoint, that is 
\[
\text{J}=(j_0,j_1,\dots,j_{p+1}), \quad \text{with} \quad 0=j_0<j_1<\dots<j_{p+1}=q.
\]

Since all breakpoint intervals contained in each $[x_{j_h},x_{j_{h+1}}]$ have same degree, that is $d_{j_h}=\dots=d_{j_{h+1}-1}$, we 
can break the target space $\SSS$ into a sequence of conventional spline spaces, each one defined on $[x_{j_h},x_{j_{h+1}}]$, and then  
join these spaces two by two.
For example, joining the two sections of $\SSS$ relative to $[x_{j_{h-1}},x_{j_{h}}]$ and $[x_{j_h},x_{j_{h+1}}]$ with continuity $k_{j_h}$ at $x_{j_{h}}$ will produce a space on the whole interval  $[x_{j_{h-1}},x_{j_{h+1}}]$, which is the restriction of the target space $\SSS$ to that interval. This join
will generate the representation matrix of the MDB-spline basis of $\SSS$
relative to the MDB-spline basis of $\SSS_0$ restricted to $[x_{j_{h-1}},x_{j_{h+1}}]$.
The resulting space can in turn be connected with the neighboring sections of $\SSS$ at $x_{j_{h-1}}$ and/or $x_{j_{h+1}}$ with continuities $k_{j_{h-1}}$ and $k_{j_{h+1}}$, respectively. 
Not that these joins must be performed in a specific order, namely from higher
continuity to lower continuity, in such a way to guarantee that, before each
repetition of Algorithm \ref{alg:1}, all the necessary information (representation matrices and integrals of the MDB-splines) relative to the derivative spaces (up to the required order of differentiation) to the right and left of the join have been generated as the output of the previous joins.

In this paper, Algorithm \ref{alg:1} mostly serves as a step-up for the derivation of the actual algorithm (see Algorithm \ref{alg:2}) which will be presented in the next section. Algorithm \ref{alg:2}, in fact, is conceptually similar to Algorithm \ref{alg:1} and will be designed starting from it.
In particular, it represents a reformulation which, although less intuitive, allows for improving the method from a computational point of view. 

\section{Stable implementation of the RKI Algorithm}
\label{sec:general_algo}
In this section we will introduce some observations that will lead us to reformulate Algorithm 1 in an alternative way, which, although less intuitive, is numerically stable and more efficient from the point of view of the calculations to be performed. 

To this end, we start by observing that \eqref{eq:KI_rev} may raise some concern about the possibile occurrence of cancellation errors, due to the differences at the numerators and denominators.
The following result shows that the RKI coefficients can indeed be determined without resorting to the differences of Greville absciss\ae, thus it overcomes the aforementioned stability issues. In addition it also improves on the computational cost of the procedure (intended as the number of
operations to be performed) with respect to using \eqref{eq:Greville} and \eqref{eq:KI_rev}.

\begin{prop}\label{prop:RKI_noint}
The setting and assumptions being the same as in Proposition \ref{prop:RKI_greville}, the
  RKI coefficients in \eqref{eq:KI_rev} can be calculated as follows: 
 \begin{equation}\label{eq:RKI_noint}
\alpha_i = \alpha_{i-1}^{(1)} \frac{\int_a^b \hat{N}_{i-1}^{(1)}(x) dx}{\int_a^b N_{i-1}^{(1)}(x) dx},  \qquad
  i=\ell-d_j+1,\ldots,\ell-d_j+k_j,
 \end{equation}
where
$\alpha^{(1)}_{i}$  are the coefficients of reverse knot insertion from  $D\widehat{\SSS}$ to $D\SSS$. 
\end{prop}

\begin{proof}
Under the above assumptions, the MDB-spline basis functions of $D\widehat{\SSS}$ and $D\SSS$ 
are such that $\widehat{N}_j^{(1)}={N}_j^{(1)}$, for $j=1,\dots,
  \ell^{(1)}-d_j^{(1)}-1=1,\dots,\ell-d_j-1$ (being $d_j^{(1)}$ the degrees in
  $D\SSS$ and $\ell^{(1)}$ the index of the largest knot in $\bs^{(1)}$ smaller
  or equal to  $x_j$). This observation and relation \eqref{eq:kiN} between the MDB-spline bases of  the derivative spaces  yield:
\begin{equation} \label{eq:dif_sum}
\begin{aligned}
\hat \xi_i - \xi_{i-1} & = a+\sum_{j=1}^{i-1} \int_a^b \hat N_j^{(1)}(x) dx  - a
  - \sum_{j=1}^{i-2} \int_a^b N_j^{(1)}(x) dx\\
&\hspace{-1.1cm}= \sum_{j=\ell- d_j}^{i-1} \int_a^b \hat N_j^{(1)}(x) dx 
  - \sum_{j=\ell- d_j}^{i-2} \int_a^b N_j^{(1)}(x) dx\\
  &\hspace{-1.1cm} = \sum_{j=\ell-d_j}^{i-1} \int_a^b \hat N_j^{(1)}(x) dx 
  - \sum_{j=\ell-d_j}^{i-2} \left(\alpha^{(1)}_{j} \int_a^b \hat{N}_j^{(1)}(x)dx
  + (1- \alpha_{j+1}^{(1)}) \int_a^b \hat{N}_{j+1}^{(1)}(x)dx \right)\hspace{-0.1cm}.
  \end{aligned}
\end{equation}
Hence the numerator of \eqref{eq:RKI_noint} comes from the above identity and the
  fact that $\alpha_{\ell-d_j}^{(1)}=1$, whereas the denominator straightforwardly follows from \eqref{eq:Greville}. 
\end{proof}

\begin{rem}\label{rem:nodiff}
Using again relation \eqref{eq:kiN} between the MDB-spline bases of  the derivative spaces
and \eqref{eq:RKI_noint}, we can obtain the following formula:
\begin{equation} \label{eq:umalfa}
  1-\alpha_{i}  =(1-\alpha_{i}^{(1)})\frac{ \int_a^b \hat{N}_{i}^{(1)}(x) dx}{\int_a^b N_{i-1}^{(1)}(x) dx}.
\end{equation}
This result avoids us to actually perform any differences of type $1-\alpha_{i}$ or $1-\alpha_{i}^{(1)}$.
In particular, when raising the continuity from $C^0$ to $C^1$, that is passing from spaces $S_0^n$ in the first column to spaces $S_1^n$ in the second column of  the triangular scheme, the coefficients  $\alpha_{i}^{(1)}$ will all be trivial (that is either zero or one) and thus so will be the differences $1-\alpha_{i}^{(1)}$. 
Hence, at each subsequent iteration, the evaluation of the right-hand side of \eqref{eq:umalfa} will just involve the calculation of the ratio of two integrals and the product by the value $1-\alpha_{i}^{(1)}$ inherited from the previous step and therefore no subtraction will need to be performed.
\end{rem}
 
 \begin{rem}\label{rem:G_vs_I}
 Relation \eqref{eq:RKI_noint}, which elegantly emerges passing through Greville absciss\ae, could alternatively be proven by induction resorting to the integral definition \eqref{def:int_rec}. The latter approach was pursued in a less general context in \cite{SW2010a} to determine the coefficients of knot insertion between two MDB-spline bases. 
 \end{rem}

 \begin{algorithm}
 {\small
\setstretch{1}
\SetAlgoLined
\KwData{$\bM_L^n$ and $\bM_R^n$, $n=0,\ldots, r$; $\bINz_L^n$ and
  $\bINz_R^n$, $n=0,\ldots,r-1$.}
   \KwResult{$\bM^n_n$, $n=0,\dots,r$.}
    \For{$n \leftarrow  0$ \KwTo $r-1$}{
      $\bIN_{L}^{n} \leftarrow \bM_{L}^{n} \cdot \bINz_L^{n}$\;
      $\bIN_{R}^{n} \leftarrow \bM_{R}^{n} \cdot \bINz_R^{n}$\;
      $\bIN^{n,0} \leftarrow [\bIN_L^{n},\bIN_R^{n}]$ \;
      $\bINz^{n} \leftarrow [\bINz_L^{n},\bINz_R^{n}]$ \;
      $\bM^{n,0} \leftarrow [\bM_L^{n},\bM_R^{n}]$ \;
      }
 $\bM^{r,0} \leftarrow [\bM_L^{r},\bM_R^{r}]$ \;
 $ibstart \leftarrow \ell^{1,1}-d_j^{1,1}+1$ \;
 \For{$n \leftarrow 1$ \KwTo $r$}{
   $ib \leftarrow ibstart$ \;
   $IN_{ib-1}^{n-1,-1} \leftarrow \;$ last element of $\bIN_L^{n-1}$ \; 
   $IN_{ib}^{n-1,-1} \leftarrow \;$ first element of $\bIN_R^{n-1}$ \; 
   \For{$k \leftarrow 1$ \KwTo $n$}{
     $\alpha_{ib-1}^{n-1,k-1} \leftarrow 1$ \;
     $\beta_{ib+k-1}^{n-1,k-1} \leftarrow 1$ \;
    \For{$i \leftarrow ib$ \KwTo $ib+k-1$}{
      $\alpha_i^{n,k} \leftarrow
      \alpha_{i-1}^{n-1,k-1} \, IN_{i-1}^{n-1,k-2}/ IN_{i-1}^{n-1,k-1}$\;
 $\beta_i^{n,k} \leftarrow \beta_{i}^{n-1,k-1} IN_{i}^{n-1,k-2}/IN_{i-1}^{n-1,k-1}$; $\quad \; // \beta_i^{n,k}$ store $1-\alpha_i^{n,k}$
    }
    $\bM^{n,k} \leftarrow \bA^{n,k} \cdot \bM^{n,k-1}$ \;
    \If {$n<r$}{
      Compute $\bIN^{n,k} \leftarrow \bM^{n,k} \cdot \bINz^{n}$\;
    }
    $ib \leftarrow ib - 1$ \;
   }
   $ibstart \leftarrow ibstart + 1$ \;
}
}
    \caption{Stable matrix representation of the $C^r$ join of two MD-spline spaces $\SSS_L$ and $\SSS_R$.
    }
\label{alg:2}
\end{algorithm}

The procedure for the $C^r$ join of two MDB-spline spaces can be revisited 
on account of the above discussion, leading to Algorithm \ref{alg:2}.
In the algorithm, as well as in the examples presented below, we indicate by $\alpha_i^{n,k}$ the RKI coefficients 
to pass from $\SSS^{n,k-1}$ to $\SSS^{n,k}$ and by $\beta_i^{n,k}$ the differences $1-\alpha_i^{n,k}$. 
Furthermore, for each row $n$ of the triangular scheme, we need to identify
the index $ibstart$ of the first nontrivial RKI coefficient to be determined. 
Its initial value (line 9) is derived from Proposition \ref{prop:RKI} applied
to spaces $\SSS^{1,0}$ and $\SSS^{1,1}$, being
$d_j^{1,1}$ the degree in the interval to the left of breakpoint $x_j$
(that is, in the algorithm $j$  is the index of the breakpoint where $\SSS_L$ and $\SSS_R$ are joined and $r$ stands for $k_j$) 
and being $\ell^{1,1}$ computed with respect to the left extended partition of
$\SSS^{1,1}$.
Subsequently, the indices of the first non-zero RKI coefficients are
determined incrementing $ibstart$ while
 $n$ increases and decrementing it while $k$ increases.

 The following example not only illustrates the application of Algorithm \ref{alg:2} on a practical case, but also 
 demonstrates how to generate the matrix representation of an arbitrary 
 MD-spline space following the genaral outline discussed at the end of section \ref{sec:RKI_rev},
that is by \qtext{breaking} the target space into a sequence of conventional spline spaces and joining these spaces two by two with the required continuities.

\begin{exmp}[Matrix representation via RKI]
\label{ex:RKI_example2}
Let us consider the target space $\SSS(\PPP_{\bd},\XXX,\KKK)$ defined on  $[0,4]$, 
with breakpoints $\XXX=\{1,2,3\}$, degrees $\bd=(2,2,4,3)$ and continuities $\KKK=(1,2,3)$.
The associated $C^0$ MD-spline space will be likewise defined on $[0,4]$, have same breakpoints $\XXX$ and degrees $\bd$ and will have continuities $\KKK_0=(1,0,0)$.

Space $\SSS$ can be seen as the join of three spaces, and more precisely of  
a degree-$2$ conventional spline space $\SSS_A$ on 
$[0,2]$, a degree-$4$ polynomial space $\SSS_B$ on $[2,3]$ and a degree-$3$
  polynomial space $\SSS_C$ on $[3,4]$. 
We shall hence apply Algorithm \ref{alg:2} twice, to generate a $C^2$ join at point $2$ and a $C^3$ join at point $3$.
As we will see, these joins should be processed starting from the one of higher continuity, since this guarantees that all the information necessary to perform a join is either trivially known or has been computed during the previous ones. 
Therefore we will first calculate the $C^3$ join of spaces $\SSS_B$ and $\SSS_C$ at $3$, and then calculate the  $C^2$ join of the resulting space with $\SSS_A $ at $2$. 

With reference to Algorithm \ref{alg:2}, in which $\SSS_L$ and $\SSS_R$ will be the spaces 
$\SSS_B$ and $\SSS_C$ of this example, the first join is described by the
  triangular scheme \eqref{eq:tri}, which we rewrite below indicating the
  degrees and continuities in each space $S^{n,k}$ in the form $(\text{degree}_{\; \text{continuity}}\text{degree})$,
along with the nontrivial RKI coefficients necessary to pass from one space to another:

\vspace{-0.5cm}
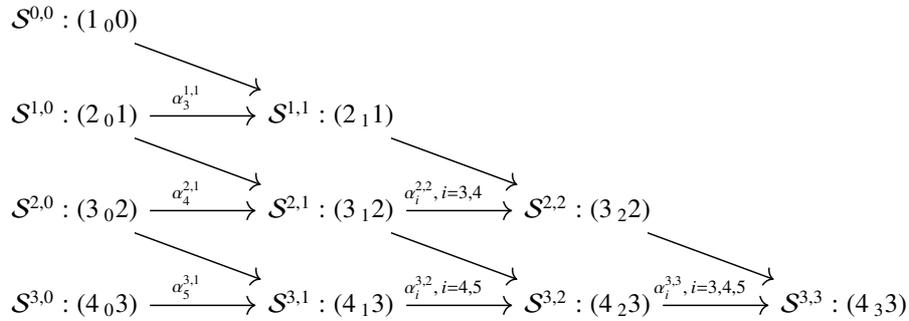
\begin{equation}
\label{eq:tri_ex1}
\end{equation}

\begin{center}
\begin{tikzcd}[column sep={4em}]
\SSS^{0,0}:(1_{\; 0}0)
\arrow{rd}[inner sep=1pt]{} 
 & & & \\
\SSS^{1,0}:(2_{\; 0}1)  
\arrow{rd}[inner sep=1pt]{}  
\arrow{r}[pos=0.25,inner sep=2pt]{\;\;\;\; \alpha_3^{1,1}}  
& 
\SSS^{1,1}: (2_{\; 1}1)  
\arrow{rd}[inner sep=1pt]{} 
& & \\ 
\SSS^{2,0}:(3_{\; 0}2)
\arrow{rd}[inner sep=1pt]{} 
\arrow{r}[pos=0.25,inner sep=2pt]{\;\;\;\; \alpha_4^{2,1}}  &
\SSS^{2,1}:(3_{\; 1}2)
\arrow{rd}[inner sep=1pt]{} 
\arrow{r}[pos=0.25,inner sep=2pt]{\;\;\;\; \alpha_i^{2,2},\, i=3,4}  & 
\SSS^{2,2}:(3_{\; 2}2)
\arrow{rd}[inner sep=1pt]{} &
\\
\SSS^{3,0}:(4_{\; 0}3)
\arrow{r}[pos=0.25,inner sep=2pt]{\;\;\;\; \alpha_5^{3,1}}  &
\SSS^{3,1}:(4_{\; 1}3) 
\arrow{r}[pos=0.25,inner sep=2pt]{\;\;\;\; \alpha_i^{3,2},\, i=4,5}  & 
\SSS^{3,2}:(4_{\; 2}3)
\arrow{r}[pos=0.25,inner sep=2pt]{\;\;\;\; \alpha_i^{3,3},\, i=3,4,5}   &
\SSS^{3,3}:(4_{\; 3}3)
\end{tikzcd}
\end{center}

\smallskip
  \noindent
  Being the $C^0$ join of two polynomial spaces,
each $\SSS^{n,0}$, $n=0,\dots,3$, is a $C^0$ MD-spline space having dimension $2(n+1)$. Hence $\bM_L^n$ and $\bM_R^n$ are identity matrices and $\bIN_L^n \equiv \bINz_L^n$,  $\bIN_R^n \equiv \bINz_R^n$.
For $n=0,\ldots,2$, the integral vectors $\bINz_L^n$, resp.\ $\bINz_R^n$, can be calculated by observing
that the basis functions in
$D^{r-n}\SSS_L$, resp.\  $D^{r-n}\SSS_R$, are conventional B-splines of degree $n+1$, resp. $n$. Hence, according
 to \eqref{eq:intC0}, $\bINz_L^n$ has $n+2$ entries equal to $1/(n+2)$
and $\bINz_R^n$ has $n+1$ entries equal to $1/(n+1)$.

The vectors $\bIN^{n,0}$, $\bINz^n$ and matrices $\bM^{n,0}$ are obtained by the previously discussed $C^0$ join operation $[\,,]$; in particular in this example
\[
  \bIN^{n,0} \equiv \bINz^n=\left(\frac{1}{n+2}, \dots, \frac{1}{n+2}+\frac{1}{n+1}, \dots,  \frac{1}{n+1}\right).
\]

\noindent
Triggering Algorithm \ref{alg:2} with this information, we obtain for $n=1$, $k=1$:
\[
\alpha_3^{1,1}=\alpha_2^{0,0} \frac{IN_2^{0,-1}}{IN_2^{0,0}}=\frac13, \quad \text{ and } \quad \beta_3^{1,1}=\beta_3^{0,0} \frac{IN_3^{0,-1}}{IN_2^{0,0}}=\frac23,\]
from which 
\[
  \bA^{1,1}=\left ( \begin{array}{cccc}
  1 & 0 & 0 & 0 \\[0.5ex]
  0 & 1 & \frac23 & 0 \\[0.5ex]
  0 & 0 & \frac13 & 1 \\
  \end{array} \right ), \qquad \bM^{1,1}=\bA^{1,1}, \qquad
\bIN^{1,1}=\bM^{1,1}\bINz^1=\left(\frac13, \frac89, \frac79\right). 
\]
For $n=2, k=1$, we obtain: 
\[\alpha_4^{2,1}=\alpha_3^{1,0} \displaystyle{ \frac{IN_3^{1,-1}}{IN_3^{1,0}}}=\frac25, \quad \text{and} \quad
\beta_4^{2,1}=\beta_4^{1,0} \displaystyle{ \frac{IN_4^{1,-1}}{IN_3^{1,0}}}=\frac35,
\]
from which 
\[
  \bA^{2,1}=\left ( \begin{array}{cccccc}
  1 & 0 & 0 & 0 & 0 & 0\\[0.5ex]
  0 & 1 & 0 & 0 & 0 & 0\\[0.5ex]
  0 & 0 & 1 & \frac35 & 0 & 0\\[0.5ex]
  0 & 0 & 0 & \frac25 & 1 & 0\\[0.5ex]
  0 & 0 & 0 & 0 & 0 & 1\\
\end{array} \right ), 
\qquad
\bM^{2,1}=\bA^{2,1},
\;
\]
and 
\[
  \bIN^{2,1}=\bM^{2,1}\bINz^2=\left(\frac14, \frac14, \frac35, \frac{17}{30}, \frac13\right).
\]
We shall then proceed to $n=2, k=2$, obtaining:
\[\alpha_3^{2,2}=\alpha_2^{1,1} { \frac{IN_2^{1,0}}{IN_2^{1,1}}}=\frac38, \qquad
\beta_3^{2,2}=\beta_3^{1,1} { \frac{IN_3^{1,0}}{IN_2^{1,1}}}=\frac58,\]
\[
\alpha_4^{2,2}=\alpha_3^{1,1} { \frac{IN_3^{1,0}}{IN_3^{1,1}}}=\frac{5}{14}, \qquad
\beta_4^{2,2}=\beta_4^{1,1} { \frac{IN_4^{1,0}}{IN_3^{1,1}}}=\frac{9}{14},\]
which yields
\[
  \bA^{2,2}=\left ( \begin{array}{ccccc}
  1 & 0 & 0 & 0 & 0 \\[0.5ex]
  0 & 1 & \frac58 & 0 & 0\\[0.5ex]
  0 & 0 & \frac38 & \frac{9}{14} & 0\\[0.5ex]
  0 & 0 & 0 & \frac{5}{14} & 1 \\
  \end{array} \right ), \; \bM^{2,2}=\bA^{2,2} \bM^{2,1}=
\left ( \begin{array}{cccccc}
  1 & 0 & 0 & 0 & 0 & 0\\[0.5ex]
  0 & 1 & \frac58 & \frac38 & 0 & 0\\[0.5ex]
  0 & 0 & \frac38 & \frac{27}{56} & \frac{9}{14} & 0\\[0.5ex]
  0 & 0 & 0 & \frac17 & \frac{5}{14} & 1 \\
\end{array} \right )\]
and 
\[\bIN^{2,2}=(1/4, 5/8, 33/56, 15/28).\]
This completes the second row of the triangular scheme. 
Proceeding in this way for $n=3$ and $k=1,2,3$,
eventually yields the matrix 
\[
  \bM^{3,3}=\bA^{3,3} \bM^{3,2}=
\left ( \begin{array}{cccccccc}
  1 & 0 & 0 & 0 & 0 & 0 & 0 & 0 \\ [0.5ex]
  0 & 1 & \frac35 & \frac{7}{20} & \frac15 & 0 & 0 & 0 \\[0.5ex]
  0 & 0 & \frac25 & \frac{27}{55} & \frac{24}{55} & \frac{4}{11} & 0 & 0 \\[0.5ex]
  0 & 0 & 0 & \frac{7}{44} & \frac{49}{165} & \frac{238}{495} & \frac{28}{45} & 0 \\  [0.5ex]  
  0 & 0 & 0 & 0 & \frac{1}{15} & \frac{7}{45} & \frac{17}{45} & 1 \\    
\end{array} \right ). \]

Recall that  Algorithm \ref{alg:2} returns as output the representation matrices 
for all derivative spaces up to differentiation order three of the $C^3$ join of  
$\SSS_B$ and $\SSS_C$. 
We shall use this information to apply the algorithm again, this time for joining with $C^2$ continuity 
the conventional spline space $\SSS_A$ on $[0,2]$ and the MD-spline space on $[2,4]$ obtained as output of the previous join. 
As for this second round, the triangular scheme  will be: 

\smallskip
\begin{center}
\begin{tikzcd}[column sep={4em}]
  \SSS^{0,0}:(0_{\; -1}0_{\; 0}2_{\; 1}1)
\arrow{rd}[inner sep=1pt]{} 
 & & & \\
  \SSS^{1,0}:(1_{\; 0}1_{\; 0}3_{\; 2}2)
\arrow{rd}[inner sep=1pt]{}  
\arrow{r}[pos=0.25,inner sep=2pt]{\;\;\;\; \alpha_3^{1,1}}  
& 
  \SSS^{1,1}:(1_{\; 0}1_{\; 1}3_{\; 2}2) 
\arrow{rd}[inner sep=1pt]{} 
& & \\ 
  \SSS^{2,0}:(2_{\; 1}2_{\; 0}4_{\; 3}3)
\arrow{r}[pos=0.25,inner sep=2pt]{\;\;\;\; \alpha_4^{2,1}}  &
  \SSS^{2,1}:(2_{\; 1}2_{\; 1}4_{\; 3}3) 
\arrow{r}[pos=0.25,inner sep=2pt]{\;\;\;\; \; \; \alpha_i^{2,2}, \, i=3,4}  & 
  \SSS^{2,2}:(2_{\; 1}2_{\; 2}4_{\; 3}3) 
\end{tikzcd}
\end{center}

\smallskip
\noindent
Denoted as usual  by $\SSS_L$ and $\SSS_R$ the two spaces to be joined, the corresponding representation matrix
$\bM_L^n$ will be the identity of size $n+2$, whereas  
$\bM_R^n=\bM^{n+1,n+1}$, being $\bM^{n+1,n+1}$, $n=0,1,2$, the output of the previous $C^3$ join. 
The integrals of the $C^0$ MDB-spline functions required by Algorithm \ref{alg:2} can be evaluated by \eqref{eq:intC0} for the left-hand side spaces $\SSS_L^n$, $n=0,1$, which gives:  
\[\bINz_L^0=(1,1) \quad \text{and} \quad \bINz_L^1=\left(\frac12,1,\frac12\right).\]
 For spaces $\SSS_R^n$, $n=0,1$, instead, the integrals may be stored as output of the first join or efficiently calculated at runtime by 
 \eqref{eq:intC0}, obtaining:
\[\bINz_R^0=\left(\frac13,\frac13,\frac56,\frac12\right) \quad \text{and} \quad \bINz_R^1=\left(\frac14,\frac14,\frac14,\frac{7}{12},\frac13,\frac13\right),\]
hence, from lines 2 and 3 of the algorithm, we will obtain $\bIN_L^n=\bINz_L^n$, $n=0,1$, and 
\[
\bIN_R^0=\left(\frac13,\frac89,\frac79\right)\quad \text{and} \quad \bIN_R^1=\left(\frac14,\frac58,\frac{33}{56},\frac{15}{28}\right).\]
Finally, vectors $\bIN^{n,0}$ and $\bINz^n$, $n=0,1$, (lines 4 and 5) and matrices
$\bM^{n,0}$, $n=0,1,2$, (line 6) are the $C^0$ join of the above quantities. 
The output of this second and last join is a matrix 
$\bM^{2,2}$ such that
\[
  \bN=\bM^{2,2}\bNz^2. 
\]
The above is the matrix representation of the MDB-spline basis of the target space $\SSS$ relative to the basis of the associated $C^0$ MD-spline space $\bNz=\bNz^2$. 

As previously mentioned, note that processing the joins from higher to lower continuity makes so that, each time, 
all the information to address the next join is available or has been computed during the previous steps.
\end{exmp}

With the previous example in mind, we can further discuss some details of our implementation.
In order to save on memory allocation, only one
matrix $\bM^{n,k}$ should be stored for each row of the triangular scheme, that is each $n=0,\dots,r$, since 
such matrices can be overwritten when moving from one column to the other.
In addition, it is unnecessary to create matrices $\bA^{n,k}$, which we merely introduced for ease of presentation, as the coefficients $\alpha_i^{n, k}$ and
$\beta_i^{n, k} $ can be stored in temporary one-dimensional arrays, to be
destroyed after been used for the coefficient computations at lines 15 and 16
and the RKI steps at lines 18 and 19. 
The integrals $\bIN^{n,k}$ can as well be stored in temporary one-dimensional arrays. 
Moreover, only one array can be used to store all vectors  $\mathbf{IN0}^n$,  
$n=0,\ldots,r-1$, since each of these can be overwritten at the end of the corresponding row of the triangular scheme. 

\section{Stability analysis}
\label{sec:stability}
Unlike how it usually happens, namely that we propose an algorithm and then we analyze its
stability, we designed an algorithm that would possess all the
characteristics to be numerically stable. 
This feature becomes clear if we  break Algorithm \ref{alg:2} into a sequence
of basic steps, each involving numerically stable operations only. The results
of this analysis will be confirmed and highlighted by the numerical
experimentation presented in subsection \ref{sec:num_cons_RKI}.

Our discussion may benefit from some preliminary considerations.
First, it is easy to count how many RKI coefficients will be calculated over the course of the algorithm. 
In particular, the \qtext{for} loops at lines 14 and 17 show that  we will have to calculate one coefficient $\alpha_i^{n,1} $ for $n = 1,\ldots r$, two coefficients $\alpha_i^{n, 2} $ 
for $n = 2,\ldots,r$ and so on up to $r$ coefficients $\alpha_i^{n,r} $ for $n = r$.
Also note that, for $k = 1$, at lines 18 and 19 reference is made to elements of
vectors $\bIN^{n-1,-1} $, never formally initialized, but
whose values are trivially known from $\bIN_L^{n}$ and
$\bIN_R^{n}$ (lines 12 and 13).
Again, for each $n$ and $k$, the first $\alpha_{i}^{n-1,k-1}$ and the last
$\beta_i^{n-1,k-1}$ considered are trivially equal to one (lines 15 and 16).
Finally, as recalled earlier, storing the quantities $\beta_i^{n-1, k-1}$ allows us to avoid the evaluation of the quantities $1-\alpha_i^{n, k}$ and therefore the whole algorithm does not contain any floating point subtraction.

Bearing in mind these observations,  we can break the algorithm in the following basic steps.
\begin{itemize}[leftmargin=0.6cm]
  \item[A)] Calculation of the input vectors $\bINz_R^n$ and $\bINz_L^n$. 
 Since functions in $\bNz_R^n$ and $\bNz_L^n$ are $C^0$ MDB-splines, the
    evaluation of their integrals involves computing and adding the
    integrals of conventional B-splines according to \eqref{eq:intC0},
    all of which are positive quantities. 
Likewise, the $C^0$ join of the integral vectors at lines 4 and 5 involves summations between positive quantities. 
   
\item[B)] Products between matrices $\bM^{n,k}$ (as well as $\bM_L^n$ and $\bM_R^n$), all of which entries belong to $[0,1]$, 
and positive vectors $\bINz^{n}$ (as well as $\bINz_L^{n}$ and $\bINz_R^{n}$) (lines 2,3 and 23).
Due to the fact that only some elements of the vectors at the right-hand side
    of these assignments are used, these products are reduced to dot products
    between single rows of matrices $\bM^{n,k} $ and vectors $\bINz^{n}$.
Note that each iteration involves as many such dot products as the integrals at lines 18 and 19, that is 3 dot products at most (since some of those integrals are used twice, so they could be stored and reused). 

  \item[C)] Evaluation of the right-hand sides of the assignments at lines 18 and 19. 
    This amounts to calculating first the product, which produces a value
    in $[0,1]$, and then the ratio, obtaining a result in $[0,1]$ as can be
    seen from the fact that 
    $N_{i-1}^{n-1,k-1}=\alpha_{i-1}^{n-1,k-1}N_{i-1}^{n-1,k-2}
    + (1-\alpha_{i}^{n-1,k-1}) N_{i}^{n-1,k-2}$.

 \item[D)] Product at the right-hand side of the assignment at line 21. 
Rather than a matrix product, it is convenient to perform this calculation as a repeated combination of two rows  
of $\bM^{n,k-1}$ (all of which entries are in $[0,1]$), of the form 
 $\alpha_j^{n,k} \mathbf{m}_{j-1}^{n,k-1} + \beta_{j+1}^{n,k} \mathbf{m}_j^{n,k-1}$,
   where $\alpha_j^{n,k}$ and  $\beta_{j+1}^{n,k}=(1-\alpha_{j+1}^{n,k})$ are
    entries on the bidiagonal of $\bA^{n,k}$ and $ \mathbf{m}_j^{n,k-1}$ is the $j$th row of
    $\bM^{n,k-1}$. 
\end{itemize}

The above analysis emphasizes that the proposed algorithm consists of summations, ratios and products between positive quantities (most of which belonging to $[0,1]$) and dot products between vectors with positive entries, all of which are numerically stable arithmetic operations (see \eg \cite{Higham2002}).
It also allows us to compute how many operations will be performed for the $C^r$ join of two MD-spline spaces, that is:
\begin{itemize}[leftmargin=0.6cm]
  \item $r$ operations of type A);
  \item ${\frac{(r-1)r}{2}}$ operations of type B), or 
    $ 3(2(r-1) + 3(r-2) + \ldots (r-1)2 + r)$ dot products;
  \item $r + 2(r-1) + 3(r-2) + \ldots (r-1) \; 2 + r$ operations of type C); 
  \item ${\frac{r(r+1)}{2}}$ operations of type D)
or $2r+3(r-1)+4(r-2)+\ldots+2 r +r +1$ combinations of two rows of
    $\bM^{n,k-1}$, that is as many as the overall number of nontrivial RKI coefficients $\alpha_i^{n,k}$ plus one;
  \end{itemize}
and thus to estimate the computational complexity
of the algorithm, which amounts to $O(r^2) $ operations. 

\subsection{Experimental results} \label{sec:num_cons_RKI}
Besides supporting the conclusions of the above stability analysis, the following numerical experiments 
provide a comparison between the new proposal  and previous ones. For the sake of brevity, we will refer to the present method and to those in  \cite{BC2020} and \cite{SPE2018} as RKI/Greville, RKI/Derivative and H-Operator, respectively. Recall that both the RKI/Derivative and H-Operator algorithms make use of derivatives (of order up to the target continuity) of MDB-splines and thus suffer in a similar way from the fact that those quantities may be very large numbers. 

Our analysis is based on calculating and comparing the algorithmic errors on the evaluation of MDB-spline basis functions and/or on the representation matrix. To this end, the \qtext{exact} values are obtained by symbolic computation, using MATLAB's Symbolic Math Toolbox, whereas the numerical results rely on MATLAB's standard precision (rounding unit $U \approx 10^{-16}$). In all the examples, the symbolic implementation of the RKI/Greville algorithm was able to produce an output within reasonable time, due to the fact that 
the method performs operations between small quantities all of which can be stored in rational form. 
As would be expected, the response times of the symbolic procedure become impractical for more complex tests. 

This section contains three experiments. The first (Example \ref{ex:Cox}) is aimed at evaluating how our analysis approach, based the algorithmic error, relates to the a posteriori error bound in Cox's seminal paper on the evaluation of B-splines \cite{Cox1972}. Like the referenced paper, this example is concerned with conventional B-splines and as a consequence the representation matrix is the identity matrix.
In the successive two experiments (Examples \ref{ex:basis} and \ref{ex:matrix}) we   
compare the
RKI/Greville Algorithm with previous proposals on a variety of test spaces featured by both uniform and nonuniform distributions of breakpoints as well as largely inhomogeneous degrees. The parameters of the different test spaces that will be considered are summarized in 
Table \ref{tab:tests}.

\begin{table}[h]
\begin{center}
\resizebox{\textwidth}{!}{
\begin{tabular}{cllll}
\hline
{}& $[a,b]$ & $\XXX$ &  $\bd$  & $\KKK$ \\
\hline
  Test 1 & $[-10000,10000]$ & $\{-9999,0,9999\}$ & $ (5,3,3,5) $ & $(3,2,3)$\\[0.5ex]
  Test 2 & $[-10000,10000]$ & $\{-9999,0,9999\}$ & $(3,5,5,3)$ & $(3,4,3)$ \\[0.5ex]
  Test 3 & $ [1,1024]$ & $\{2^j\},\; j=1,\dots,9$ & $(9,9,10,10,9,9,10,10,9,9)$ & $(8,9,9,9,8,9,9,9,8)$\\[0.5ex]
  Test 4 & $[-1024,1]$ & $\{-2^{10-j}\}, \; j=1,\dots,9$ & $( 9,9,10,10,9,9,10,10,9,9) $ & $(8,9,9,9,8,9,9,9,8)$ \\[0.5ex]
  Test 5 & $[0,22]$ & $\{j\}, \; j=1,\ldots,21$ & $d_i=19$, $i=10,\dots,11$; &  $k_i=18$, $i=11,\dots,12$; \\
 & & & $d_i=20$, $i=5,\dots,9,12,\dots,16$; &  $k_i=19$ $i=6,\dots,10,13,\dots,17$; \\
 &  &  & $d_i=21$, $i=0,\ldots,4,17,\dots,21$ &  $k_i=20$, $ i=1,\ldots,5,18,\dots,21$;
 \\[0.5ex]
  Test 6 & $[-10000,10000]$ &  $\{-9999,0,9999\}$ & $(21,19,19,21) $ & $(15,10,15)$\\
\hline
\end{tabular}
}
  \caption{Test spaces for Examples \ref{ex:basis} and \ref{ex:matrix}.}
  \label{tab:tests}
\end{center}
\end{table}

\begin{exmp}[A comparison with conventional B-splines] \label{ex:Cox}
This experiment replicates \cite[Example 2]{Cox1972}, which is the most challenging test in the referenced paper. The setting is a conventional spline space of degree 21, defined in the interval $[0,22]$, with equispaced breakpoints $x_j$ placed at the integers and $C^{20}$ continuity at each breakpoint. Note that choosing both the breakpoints and the evaluation points to be exactly represented in the floating point standard  allows for avoiding roundoff errors in the initial data. 
Table \ref{tab:Cox} shows the algorithmic error on the evaluation of the \qtext{central} B-spline $N_{22,21}$ at the breakpoints $x_j$.
For the same experiment, \cite[Table 2]{Cox1972} reports the values of $N_{22,21}$ along with the a posteriori error bounds established in that paper. 
In particular, the values of $N_{22,21}$ found by Cox refer to non-normalized basis functions and are the same  
as in the second column of Table \ref{tab:Cox}, whereas the values in the third
  column of Table \ref{tab:Cox} are obtained with the 
 recurrence relation for normalized $C^0$ MDB-splines in \cite{BC2020}, which is a simple generalization of the more established scheme in \cite{dB1972}. The values in the two columns, however, only differ by a 
normalization constant equal to the width of the support.

\begin{table}[h]
\begin{center}
\resizebox{\textwidth}{!}{
\begin{tabular}{rccccc}
\hline
  $x_j$ &  $N_{22,21}(x_j)$ Non-Normalized & $N_{22,21}(x_j)$ Normalized & Error Bound & Absolute
  Alg. Error & Relative Alg. Error\\
\hline
  1 &   8.896791392450574e-22&    1.957294106339126e-20&  3.2378e-34 & 1.3644e-36 & 6.9706e-17 \\
  2&    1.865772813284987e-15&    4.104700189226971e-14&  6.7901e-28 & 4.0230e-30 & 9.8009e-17 \\
  3 &   9.265310806863227e-12&    2.038368377509910e-10&  3.3719e-24 & 3.3787e-26 & 1.6575e-16 \\
  4&    3.708541354285271e-09&    8.158790979427597e-08&  1.3497e-21 & 7.0656e-24 & 8.6601e-17 \\
  5&    3.402962627063746e-07&    7.486517779540241e-06&  1.2384e-19 & 9.0997e-23 & 1.2155e-17 \\
  6&    1.107329203006056e-05&    2.436124246613324e-04&  4.0299e-18 & 2.4981e-20 & 1.0254e-16 \\
  7&    1.595958078468785e-04&    3.511107772631326e-03&  5.8082e-17 & 9.0643e-19 & 2.5816e-16 \\
  8&    1.156908330166488e-03&    2.545198326366273e-02&  4.2103e-16 & 3.6835e-18 & 1.4472e-16 \\
  9&    4.554285942496692e-03&    1.001942907349272e-01&  1.6574e-15 & 7.1213e-18 & 7.1075e-17 \\
  10&    1.019454972176512e-02&    2.242800938788327e-01&  3.7101e-15 & 2.1568e-17 & 9.6165e-17 \\
  11&    1.330103123779249e-02&    2.926226872314347e-01&  4.8407e-15 & 8.2012e-17 & 2.8026e-16 \\
  12&    1.019454972176512e-02&    2.242800938788327e-01&  3.7101e-15 & 2.1568e-17 & 9.6165e-17 \\
  13&    4.554285942496692e-03&    1.001942907349272e-01&  1.6574e-15 & 7.1213e-18 & 7.1075e-17 \\
  14&    1.156908330166488e-03&    2.545198326366273e-02&  4.2103e-16 & 3.6835e-18 & 1.4472e-16 \\
  15&    1.595958078468785e-04&    3.511107772631326e-03&  5.8082e-17 & 9.0643e-19 & 2.5816e-16 \\
  16&    1.107329203006056e-05&    2.436124246613324e-04&  4.0299e-18 & 2.4981e-20 & 1.0254e-16 \\
  17&    3.402962627063746e-07&    7.486517779540241e-06&  1.2384e-19 & 9.0997e-23 & 1.2155e-17 \\
  18&    3.708541354285271e-09&    8.158790979427597e-08&  1.3497e-21 & 7.0656e-24 & 8.6601e-17 \\
  19&    9.265310806863227e-12&    2.038368377509910e-10&  3.3719e-24 & 3.3787e-26 & 1.6575e-16 \\
  20&    1.865772813284987e-15&    4.104700189226971e-14&  6.7901e-28 & 4.0230e-30 & 9.8009e-17 \\
  21&    8.896791392450574e-22&    1.957294106339126e-20&  3.2378e-34 & 1.3644e-36 & 6.9706e-17 \\
\hline
\end{tabular}}
  \caption{Numerical experiments reported in Example \ref{ex:Cox}}
  \label{tab:Cox}
\end{center}
\end{table}

A running error analysis was also integrated in our implementation and returned a posteriori error bounds in accordance with those reported by Cox (considering that we work in double precision with 16 digits, while Cox with 11 digits).
It shall be noted, in particular, how the results in the column of absolute algorithmic errors are consistent with the corresponding error bounds and the corresponding relative errors that will be used to assess the numerical stability of our proposal.

In the conclusions of \cite{Cox1972}, on the basis of the a posteriori error bound, it is expected that the maximum relative error attained with a $t$-digits mantissa cannot exceed $(70)2^{- t}$ for degree $10$ or less, whereas it cannot exceed $(700)2^{- t}$ for degree $100$ or less. It is also observed that the bound on the relative error grows linearly with the degree of a spline.
Our experimentation shows that the actual error is even
lower. In fact, for degree $100$ or less the relative algorithmic error for most experiments is about $10^{- 16}$, with only a few values of the order of $ 10^{- 15}$, whereas the bound estimated by Cox would be of the order of $10^{- 14}$.
We believe that this may be attributable to cancellation of rounding error, 
which may cause the final computed answer to be much more accurate than the intermediate quantities. 
This phenomenon has been described, \eg, in \cite[p.19]{Higham2002}.

We conclude by mentioning that a similar study of algorithmic errors was carried out on the evaluation of derivatives. 
Also in this case for splines of degree less than or equal to 50 and order of differentiation up to ten we never encountered algorithmic errors exceeding $\approx 10^{- 14}$.
\end{exmp}

\vspace{7pt}
\noindent

\begin{exmp}[Algorithmic error on the evaluation of MDB-splines]\label{ex:basis}
This experiment illustrates how erroneous the results of the RKI/Derivative method can be for degrees as low as three and five if the knot spacing is highly nonuniform.
Such a case is important in practice since it is often of interest to investigate the case of near-coincident knots.  
  From Table \ref{tab:ex:2} one can observe that at 
  $x_1 = -9999$ and $x_3= 9999$ the values calculated by RKI/Greville agree for symmetry, while this is not the case for the corresponding results obtained by RKI/Derivative. Moreover, the values of the algorithmic errors show that the accuracy of the RKI/Derivative method is limited to the first 6/7 digits of precision, as appearing from the value of the central MDB-spline for $x_2 = 0$. 
Similar results are reported in Table \ref{tab:ex:3}, from which one can again see that the RKI/Derivative method returns strongly asymmetric results despite the expected symmetry of the evaluated  MDB-spline. 
In both experiments the results obtained by RKI/Greville agree for symmetry and are extremely accurate, which is consistent with the conclusions of the theoretical analysis. 
  
The experiment reported in Table \ref{tab:ex:4} concerns a space with a less challenging uneven distribution of breakpoints, but higher degrees. In this case, the RKI/Derivative method appears adequate up to 9 figures only. 
Analogous results were also obtained for the spaces \qtext{Test 4} , \qtext{5} and \qtext{6}. 
Overall, the large errors for the RKI/Derivative algorithm show that the method is potentially unstable.
Conversely the small algorithmic errors of the RKI/Greville method confirm its stability.
The same conclusions are supported by the results illustrated in Example \ref{ex:matrix}, concerned with the algorithmic errors with respect to the entries of the representation matrices. 

\begin{table}[h]
\begin{center}
\resizebox{\textwidth}{!}{
\begin{tabular}{ccccc}
\hline
\multicolumn{1}{c}{\multirow{2}{*}{}} &
\multicolumn{2}{c}{$\text{RKI/Greville}$} &
\multicolumn{2}{c}{$\text{RKI/Derivative}$} \\
  \cline{2-3} \cline{4-5}
 x & $N_{5}(x)$ Normalized & Relative Alg. Error & $N_{5}(x)$ Normalized & Relative Alg. Error \\
\hline
 -9.999000e+03 &  4.500275008083014e-09 &   1.8381e-16 &   4.500275772672185e-09 &   1.6990e-07 \\
  0.000000e+00 &  5.000083333610773e-01 &   0.0000e+00 &   5.000084045999867e-01 &   1.4248e-07 \\
  9.999000e+03 &  4.500275008083015e-09 &   0.0000e+00 &   4.500275649258610e-09 &   1.4247e-07 \\
\hline
\end{tabular}}
  \caption{Numerical results discussed in Example \ref{ex:basis} for space
  \qtext{Test 1} in Table \ref{tab:tests}.}
  \label{tab:ex:2}
\end{center}
\resizebox{\textwidth}{!}{
\begin{tabular}{ccccc}
\hline
\multicolumn{1}{c}{\multirow{2}{*}{}} &
\multicolumn{2}{c}{$\text{RKI/Greville}$} &
\multicolumn{2}{c}{$\text{RKI/Derivative}$} \\
  \cline{2-3} \cline{4-5}
 x & $N_{4}(x)$ Normalized & Relative Alg. Error & $N_{4}(x)$ Normalized & Relative Alg. Error \\
\hline
  -9.999000e+03 &  2.499250262410031e-12 &   0.0000e+00 &   2.499214206146373e-12 &   1.4427e-05 \\
  0.000000e+00 &  3.750749868799358e-01  &  0.0000e+00 &   3.750749863390447e-01  &  1.4421e-09 \\
  9.999000e+03 &  2.499250262410030e-12 &   1.6161e-16 &   2.499250262410031e-12 &   1.6161e-16 \\
\hline
\end{tabular}}
  \caption{Numerical results discussed in Example \ref{ex:basis} for space
  \qtext{Test 2} in Table \ref{tab:tests}.}
 \label{tab:ex:3}
  \resizebox{\textwidth}{!}{
  \begin{tabular}{ccccc}
\hline
\multicolumn{1}{c}{\multirow{2}{*}{}} &
\multicolumn{2}{c}{$\text{RKI/Greville}$} &
\multicolumn{2}{c}{$\text{RKI/Derivative}$} \\
  \cline{2-3} \cline{4-5}
 x & $N_{9}(x)$ Normalized & Relative Alg. Error & $N_{9}(x)$ Normalized & Relative Alg. Error \\
\hline
  2.000000e+00  &  2.912087112938504e-13 &   3.4674e-16 &   2.912087106308203e-13 &   2.2768e-09 \\
  4.000000e+00 &   1.275774160308294e-09 &   1.6209e-16 &   1.275774157784237e-09 &   1.9785e-09 \\
  8.000000e+00 &   4.806036147184862e-07 &   2.2030e-16 &   4.806036141267946e-07 &   1.2311e-09 \\
  1.600000e+01 &   5.258129295850228e-05 &   3.8662e-16 &   5.258129293072319e-05 &   5.2831e-10 \\
  3.200000e+01 &   2.147713272383253e-03 &   8.0771e-16 &   2.147713271996800e-03 &   1.7994e-10 \\
  6.400000e+01 &   3.541058939374863e-02 &   5.8787e-16 &   3.541058939374988e-02 &   3.4684e-14 \\
  1.280000e+02 &   2.206016671195212e-01 &   3.7745e-16 &   2.206016671340502e-01 &   6.5860e-11 \\
  2.560000e+02 &   3.592347216925473e-01 &   0.0000e+00 &   3.592347217235125e-01 &   8.6198e-11 \\
  5.120000e+02 &   4.466585515804859e-02 &   1.5535e-16 &   4.466585516215183e-02 &   9.1866e-11 \\
\hline
\end{tabular}}
  \caption{Numerical results discussed in Example \ref{ex:basis} for space
  \qtext{Test 3} in Table \ref{tab:tests}.}
   \label{tab:ex:4}
\end{table}
\end{exmp}

\begin{exmp}[Algorithmic error on the representation matrix]\label{ex:matrix}
In this second type of test, we consider the algorithmic error on the representation matrix, calculated as
$$
\| \bM_{\text{16\_digits}} - \bM_{\text{exact}} \|_1,
$$
where $\bM_{\text{16\_digits}}$ is the numerically calculated matrix, whereas $\bM_{\text{exact}}$ is the one obtained by symbolic computation.
This is both an absolute and a relative error on account of the fact that  $\| \bM_ {\text{exact}} \|_1 = 1$.

The algorithmic error obtained with RKI/Greville is compared with those relative to both RKI/Derivative and H-operator (the  code for the latter is taken from \cite{SPE2018}.\footnote{
The H-operator algorithm, as implemented in \cite{SPE2018}, returns a representation matrix with respect to a sequence conventional B-spline bases connected  with $C^{- 1}$ continuity, therefore, with respect to ours, it has replicated columns that have been removed for a fair comparison.}).
\vspace{0.5cm}
\begin{table}[h]
\begin{center}
\resizebox{\textwidth}{!}{
\begin{tabular}{cccccccc}
\hline
  \multicolumn{1}{c}{$\text{Examples}$} &
\multicolumn{3}{c}{$\text{Algorithmic Errors}$} &
  \multicolumn{1}{c}{$\text{Examples}$} &
\multicolumn{3}{c}{$\text{Algorithmic Errors}$} \\
  \cline{2-4} \cline{6-8}
  & RKI/Greville & RKI/Derivative & H-Operator &
  & RKI/Greville & RKI/Derivative & H-Operator \\
\hline
   Test 1 & 1.0$\times$$10^{-16}$ & 2.8$\times$$10^{-7}$& 2.8$\times$$10^{-7}$ &
   Test 4 & 6.0$\times$$10^{-16}$ & 7.6$\times$$10^{-13}$& 1.1$\times$$10^{-12}$ \\
\hline
    Test 2 & 6.7$\times$$10^{-16}$ &  4.3$\times$$10^{-9}$& 4.3$\times$$10^{-9}$ &
   Test 5 & 1.0$\times$$10^{-15}$ &  5.4$\times$$10^{-2}$& 1.2$\times$$10^{-1}$ \\
\hline
   Test 3 & 3.7$\times$$10^{-16}$ & 1.1$\times$$10^{-8}$& 1.1$\times$$10^{-8}$ &
   Test 6 & 1.7$\times$$10^{-14}$ & 6.5$\times$$10^{+7}$& 6.5$\times$$10^{+7}$ \\
\hline
\end{tabular}}
\caption{Algorithmic errors on the representation matrix for the test spaces in
  Table \ref{tab:tests}.}
\label{tab:RKI_2}

\vspace{0.3cm}
\resizebox{\textwidth}{!}{
\begin{tabular}{cccccccccc}
\hline
\multicolumn{1}{c}{\multirow{2}{*}{$k_1$}} &
\multicolumn{1}{c}{\multirow{2}{*}{$K$}} &
\multicolumn{3}{c}{$\text{Algorithmic Errors}$} &
\multicolumn{1}{c}{\multirow{2}{*}{$k_1$}} &
\multicolumn{1}{c}{\multirow{2}{*}{$K$}} &
\multicolumn{3}{c}{$\text{Algorithmic Errors}$} \\
  \cline{3-5} \cline{8-10}
  & & RKI/Greville & RKI/Derivative & H-Operator &
  & & RKI/Greville & RKI/Derivative & H-Operator \\
\hline
  5 & 17 & 2.5$\times$$10^{-16}$
    &  2.5$\times$$10^{-16}$& 2.5$\times$$10^{-16}$ &
  13 & 25 & 2.7$\times$$10^{-16}$
     & 1.4$\times$$10^{-12}$& 1.4$\times$$10^{-12}$ \\
\hline
  7 & 19 & 2.2$\times$$10^{-16}$ & 1.4$\times$$10^{-14}$&
    1.4$\times$$10^{-14}$ &
  15 & 27 & 4.4$\times$$10^{-16}$
     & 2.4$\times$$10^{-11}$& 2.4$\times$$10^{-11}$ \\
\hline
  9 & 21 & 3.9$\times$$10^{-16}$
    &  4.7$\times$$10^{-14}$& 4.7$\times$$10^{-14}$ &
  17 & 29 & 3.1$\times$$10^{-16}$
     & 2.2$\times$$10^{-10}$ & 2.2$\times$$10^{-10}$ \\
\hline
  11 & 23 & 2.5$\times$$10^{-16}$
     & 1.7$\times$$10^{-13}$& 1.7$\times$$10^{-13}$ &
  19 & 31 & 4.5$\times$$10^{-16}$
     & 1.3$\times$$10^{-9}$ & 1.3$\times$$10^{-9}$ \\
\hline
\end{tabular}}
\caption{Target space $\SSS(\PPP_{\bd}, \XXX, \KKK)$
  with $[a,b]=[0,2]$, $\XXX=(1)$ and $\bd=(19,20)$; the dimension of
  $\SSS_0$ is $K_0=40$.}
  \label{tab:conf-RKI_1}
\end{center}
\end{table}

\noindent
Table \ref{tab:RKI_2} contains the algorithmic errors
obtained for all the test spaces in Table \ref{tab:tests}. 
In particular, space \qtext{Test 5} is the multi-degree counterpart of the aforementioned experiment \cite[Example 2]{Cox1972}. \qtext{Test 6}, instead, is aimed at comparing the considered algorithms in case of a very nonuniform partition and high degrees.
Finally, Table \ref{tab:conf-RKI_1} shows the algorithmic errors obtained in a test case presented in our previous paper \cite{BC2020}.
All the results confirm the adequacy of the new proposal, by contrast with previous methods, which, in some cases, suffer from serious loss of accuracy.
\end{exmp}

\section{Matrix representation in terms of the conventional B-spline basis of
maximum degree}\label{sec:RDE}
For a given target space 
 $\SSS(\PPP_{\bd}, \XXX, \KKK)$, another way to compute a matrix representation
 \eqref{eq:M} 
is to choose as initial space a conventional spline space $\SSS_0 \equiv \SSS(\PPP_ {\bd_0},\XXX, \KKK_ {0})$, with $d_j^0= m$ for all $j$,
being $m \coloneqq \max_j \{d_j \} $. In this setting, it still holds that  $\SSS \subset \SSS_0 $, but this time matrix $\bM$ needs to be computed performing successive steps of \emph{reverse degree elevation} (RDE).
As the name suggests, reverse degree elevation is the reverse operation of
degree elevation and we can understand it by referring to Remark \ref{rem:KIvsRKI}, where instead of decreasing/increasing the continuity at a breakpoint, one increases/decreases the degree on a breakpoint interval.
Therefore, each round of reverse degree elevation diminishes by one the degree in an interval, until each interval $[x_j,x_{j+1}]$ reaches the target degree $d_j$.
Overall, the number of steps $g$ required to pass from $\SSS_0$ to $\SSS$ amounts to the total number of RDE steps to be performed, that is
$ g \coloneqq \sum_{j=0}^q (m - d_j).  $
The process must be accomplished in such a way to generate a sequence of
MD-spline spaces $\SSS_n\equiv \SSS(\PPP_{\bd_n}, \XXX, \KKK_{n})$, $n=0,\dots,g$, such that
\begin{equation} \label{eq:nested2}
\SSS\coloneqq \SSS_g \subset \SSS_{g-1} \subset \cdots \subset \SSS_1 \subset \SSS_0,
\end{equation}
where each space $\SSS_n$ is defined on $[a,b]$, has same breakpoints $\XXX$ and continuities $\KKK$ as the target
space $\SSS$
and has dimension $K_n\coloneqq K+(g-n)$, being $K$ the dimension of $\SSS$.
In general, there may be more than one sequence \eqref{eq:nested2} leading from
$\SSS_0$ to $\SSS_g$ and therefore, while $\SSS_0$ and $\SSS_g$ are fixed, the
intermediate spaces $\SSS_1, \dots, \SSS_{g-1}$ will depend on the specific ordering of RDE steps performed.

\cite[Proposition 7]{BC2020} provides a result akin to Proposition \ref{prop:RKI_greville}, where space $\widehat\SSS$ is obtained from $\SSS$ through (local) degree elevation. 
In this case, the respective MDB-spline bases satisfy a relationship analogous
to \eqref{eq:alpha_ki}, with coefficients $\alpha_i$ given by 
\eqref{eq:alpha2}, the only difference being that the nontrivial coefficients $\alpha_i\in \, ]0,1[$ correspond to $i=\ell-d_j+1,\ldots,\ell$. 
These coefficients can still be determined through \eqref{eq:KI_rev}, where $\hat \xi_j$ and $\xi_j$ are the Greville absciss\ae\, of $\widehat \SSS$ and $\SSS$, respectively.

On account of Proposition \ref{prop:Greville}, the computation of the Greville absciss\ae\, of $\widehat \SSS$ and $\SSS$ entails integrating the 
MDB-spline bases of the respective derivative spaces. Hence, an RDE step can be
described by the following triangular block, akin to \eqref{eq:basic_block}:
\vspace{-0.1cm}
\begin{center}
\begin{tikzcd}[column sep={3em}]
D\SSS 
\arrow{rd}[inner sep=1pt]{G} 
 &  \\
\hspace{-0.2cm}\widehat{\SSS}    
\arrow{r}[pos=0.25,inner sep=2pt]{\;\;\;\; RDE}  
& \SSS
\end{tikzcd}
\end{center}
\vspace{-1.6cm}
\begin{equation} \label{eq:basic_blockRDE}
\end{equation}

\vspace{1.1cm}
Repeated applications of the above basic block,  
give rise to the following  \emph{rhomboid} scheme, which is the RDE counterpart of \eqref{eq:tri}, and in which $\SSS^{k,n}$, for $k=0,\ldots,r$, $n=0,\ldots,g$,   indicate the derivative spaces $D^{r-k}\SSS_n$ with 
$r\coloneqq \max \{1,\max\{k_i \in \KKK \} \}$:

\vspace{1.8cm}
\begin{equation}
\label{eq:scheme1}
\end{equation}

\vspace{-3cm}
\begin{center}
\begin{tikzcd}[column sep={2.3em}]
  \SSS^{0,0} 
  \arrow{rd}[inner sep=1pt]{G}  
  \arrow{r}[pos=0.25,inner sep=2pt]{\;\;\;\; RDE}  
  & \SSS^{0,1} 
    \arrow{rd}[inner sep=1pt]{G}  
  \arrow{r}[pos=0.25,inner sep=2pt]{\;\;\;\; RDE}
  & \cdots & \SSS^{0,g-1} 
    \arrow{rd}[inner sep=1pt]{G}  
  \arrow{r}[pos=0.25,inner sep=2pt]{\;\;\;\; RDE}
  & \SSS^{0,g} 
  \arrow{rd}[inner sep=1pt]{G}  
  & & & \\
    & \cdots   
    \arrow{r}[pos=0.25,inner sep=2pt]{\;\;\;\; RDE}   
    \arrow{rd}[inner sep=1pt]{G}  
    & \cdots   \arrow{rd}[inner sep=1pt]{G}    \arrow{r}[pos=0.25,inner sep=2pt]{\;\;\;\; RDE}
    & \cdots  & \cdots   \arrow{r}[pos=0.25,inner sep=2pt]{\;\;\;\; RDE}  \arrow{rd}[inner sep=1pt]{G}   & \cdots  \arrow{rd}[inner sep=1pt]{G}  &  & \\
    & & \SSS^{r-1,0} \arrow{rd}[inner sep=1pt]{G}  
\arrow{r}[pos=0.25,inner sep=2pt]{\;\;\;\; RDE}  
& \SSS^{r-1,1} 
\arrow{rd}[inner sep=1pt]{G}    \arrow{r}[pos=0.25,inner sep=2pt]{\;\;\;\; RDE}
& \cdots 
& \SSS^{r-1,g-1} 
\arrow{rd}[inner sep=1pt]{G}  
\arrow{r}[pos=0.25,inner sep=2pt]{\;\;\;\; RDE}  
& \SSS^{r-1,g} 
\arrow{rd}[inner sep=1pt]{G}  
& \\
    & & & \SSS^{r,0} 
    \arrow{r}[pos=0.25,inner sep=2pt]{\;\;\;\; RDE} 
    & \SSS^{r,1}   \arrow{r}[pos=0.25,inner sep=2pt]{\;\;\;\; RDE}
    & \cdots 
    & \SSS^{r,g-1} 
    \arrow{r}[pos=0.25,inner sep=2pt]{\;\;\;\; RDE} 
    & \SSS^{r,g} \\
\end{tikzcd}
\end{center}

\vspace{-0.4cm}
Note that spaces $\SSS^{0,n}$ are $C^0$ MD-spline spaces and may feature breakpoints with negative continuities, as well as intervals with negative degrees.
These correspond to the degenerate spaces involved in the generation of the MDB-spline basis of 
$\SSS^{r,n}$ by the integral recurrence relation \eqref{def:int_rec}.
Spaces $\SSS^{k,0}$ are instead conventional spline spaces of degree $m-(r-k)$.
For all spaces $\SSS^{0,n}$ and $\SSS^{k,0}$ the MDB-spline basis functions, as
well as their integrals, can be straightforwardly computed by standard
approaches, as discussed in section \ref{sec:C0}. 

Using the rhomboid scheme and the corresponding matrix representations leads to
Algorithm \ref{alg:5}, where $K(k,n)$ indicates the dimension of space
$\SSS^{k,n}$, that is 
$K(k,n)=K-(r-k)+(g-n)$, being $K$ the dimension of the target space $\SSS\equiv\SSS^{r,g}$.
The algorithm requires as input the vectors $\bIN^{k,0}$, $k=0,\ldots,r$, containing the integrals of the
  conventional B-spline bases of spaces $\SSS^{k,0}$ and the vectors $\bIN^{0,n}$,
  $n=0,\ldots,g$, of the integrals of the $C^0$ MDB-splines of the spaces $\SSS^{0,n}$. It returns as output 
  the matrices $\bM^{k,g}$ such that $\bN^{k,g} = \bM^{k,g} \bN^{k,0}$,
  $k=0,\dots,r$, where $\bN^{k,0}$ is a conventional B-spline basis of degree 
  $m-r+k$. In particular $\bM^{r,g}$ is the matrix representation of the MDB-spline 
  basis $\bN^{r,g}$ of the target space $\SSS_g\equiv\SSS^{r,g}$ with respect to
  the B-spline basis $\bN^{r,0}$ of the conventional spline space $\SSS_0
  \equiv \SSS^{r,0}$ of degree $m\coloneqq \max_i\{d_i\}$.   
We remark that, while the RKI Algorithm joins two spaces at a time, the RDE works
globally, \ie by carrying out a sequence of reverse degree elevations on all the intervals involved.

\begin{algorithm}
{\small
\setstretch{1}
\SetAlgoLined
  \KwData{$\bIN^{k,0}$, $k=0,\ldots,r$; $\bIN^{0,n}$, $n=0,\ldots,g$.}
  \KwResult{$\bM^{k,g}$, $k=1,\ldots,r$.}
 \For{$k \leftarrow 1$ \KwTo $r$}{
   $\bM^{k,0}\leftarrow\textrm{I}_{K(k,0)}$ \;
     $n \leftarrow 0$ \;
     \For{$j \leftarrow 0$ \KwTo $q$}{
     \For{$h \leftarrow m-1$ \KwTo $d_j$}{
       $n \leftarrow n+1$ \;
     $ie \leftarrow d_0^{k,n}+1+\sum_{h=1}^j d_h^{k,n}-k_{h}^{k,n}$ \;
     $ib \leftarrow ie-d_j^{k,n}+1$ \;
     $\alpha_{ib-1}^{k-1,n} \leftarrow 1$ \; 
     $\beta_{ie}^{k-1,n} \leftarrow 1$ \;
     \For{$i \leftarrow ib$ \KwTo $ie$}{
       \eIf {$k==1$}{
 $\alpha_i^{1,n} \leftarrow 
       \left(\sum_{h=ib-1}^{i-1}IN_{h}^{0,n-1}-\sum_{h=ib-1}^{i-2}IN_{h}^{0,n}\right)/IN_{i-1}^{0,n}$ \;
       $\beta_i^{1,n} \leftarrow 
       \left(\sum_{h=ib-1}^{i-1}IN_{h}^{0,n}-\sum_{h=ib-1}^{i-1}IN_{h}^{0,n-1}\right)/IN_{i-1}^{0,n}$ \;
     }{
       $\alpha_i^{k,n} \leftarrow \alpha_{i-1}^{k-1,n} {IN_{i-1}^{k-1,n-1}}/{IN_{i-1}^{k-1,n}}$ \;
       $\beta_i^{k,n} \leftarrow \beta_{i}^{k-1,n} {IN_{i}^{k-1,n-1}}/{IN_{i-1}^{k-1,n}}$ ; \hspace{0.4cm}//$\beta_i^{k,n}$ store $1-\alpha_i^{k,n}$
     }
     }
    $\bM^{k,n} \leftarrow \bA^{k,n} \bM^{k,n-1}$ \;
    \If {$k<r$}{
      $\bIN^{k,n} \leftarrow \bM^{k,n} \cdot \bIN^{k,0} $ \;
    }
   }
   }
}
}
    \caption{Matrix representation relative to the conventional B-spline basis
    of degree $m \coloneqq \max_i \{d_i\}$, with 
      $r\coloneqq \max\{1, \max\{k_i \in \KKK\}\}$
  and $g \coloneqq \sum_{j=0}^q (m - d_j)$.
  }
\label{alg:5}
\end{algorithm}

\begin{exmp}[Matrix representation via reverse degree elevation] \label{ex:RDE_example1}
In the interval $[0,3]$, let us consider the MD-spline space
  $\SSS(\PPP_{\bd},\XXX,\KKK)$ with $\XXX=\{1,2\}$, $\bd=(4,2,3)$ and $\KKK=(2,1)$.
  Let us also consider the spaces $\SSS_2$, $\SSS_1$ and $\SSS_0$ defined on the
  same interval and having same breakpoint sequence and
  continuities as $\SSS \equiv \SSS_3$ and such that
  $\SSS_3 \subset \SSS_2 \subset \SSS_1 \subset \SSS_0$.
  In particular we take 
  $\SSS_0$ to be the MD-spline space having $\bd_0=(4,4,4)$, 
  $\SSS_1$ having $\bd_1=(4,3,4)$, $\SSS_2$ having $\bd_2=(4,2,4)$
  and $\SSS_3$ having $\bd_3=(4,2,3)$.
  Note that $\SSS_0$ is a conventional spline space and hence its B-spline
  basis and corresponding integrals can efficiently be computed by known methods.
In this way one can pass from the MDB-spline basis of $\SSS_0$ 
 to that of $\SSS_1$, from that of $\SSS_1$ to that of $\SSS_2$ and finally
 from the MDB-spline basis of $\SSS_2$ to that of $\SSS_3$ performing three successive rounds of RDE. 

The rhomboid scheme of spaces in this example is as follows, where $\SSS_n\equiv\SSS^{2,n}$, $n=0,\dots,3$:

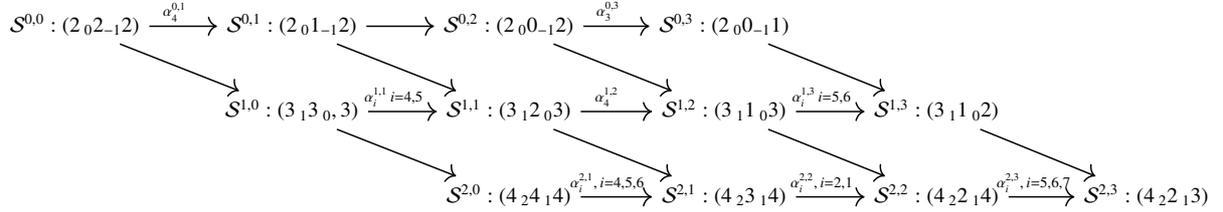
\vspace{-0.3cm}
\medskip
\begin{center}
\begin{tikzcd}[column sep={2.5em},scale cd=0.8]
  \SSS^{0,0}:(2_{\; 0}2_{-1}2) 
 \arrow{rd}[inner sep=1pt]{} 
 \arrow{r}[pos=0.25,inner sep=2pt]{\;\;\;\; \alpha_{4}^{0,1} }  
  & \SSS^{0,1}:(2_{\; 0}1_{-1}2) 
  \arrow{r}[pos=0.25,inner sep=2pt]{\;\;\;\; } 
  \arrow{rd}[inner sep=1pt]{}  
  & \SSS^{0,2}:(2_{\; 0}0_{-1}2)
  \arrow{rd}[inner sep=1pt]{}  
    \arrow{r}[pos=0.25,inner sep=2pt]{\;\;\;\;  \alpha_{3}^{0,3} }  
   & \SSS^{0,3}:(2_{\; 0}0_{-1}1) 
   \arrow{rd}[inner sep=1pt]{}  
   & & \\ 
    & \SSS^{1,0}:(3_{\; 1}3_{\;0},3) 
      \arrow{r}[pos=0.25,inner sep=2pt]{\;\;\;\; \alpha_i^{1,1}  \, i=4,5}  
    \arrow{rd}[inner sep=1pt]{}  
    & \SSS^{1,1}:(3_{\; 1}2_{\;0} 3) 
      \arrow{r}[pos=0.25,inner sep=2pt]{\;\;\;\; \alpha_4^{1,2}  }  
    \arrow{rd}[inner sep=1pt]{}  
    & \SSS^{1,2}:(3_{\; 1}1_{\;0} 3) 
      \arrow{r}[pos=0.25,inner sep=2pt]{\;\;\;\; \alpha_i^{1,3}  \, i=5,6}  
    \arrow{rd}[inner sep=1pt]{}  
    & \SSS^{1,3}:(3_{\; 1}1_{\;0} 2) 
    \arrow{rd}[inner sep=1pt]{}  
    & \\ 
    & & \SSS^{2,0}:(4_{\; 2}4_{\;1} 4) 
      \arrow{r}[pos=0.25,inner sep=2pt]{\;\;\;\; \alpha_i^{2,1}, \, i=4,5,6}  
    & \SSS^{2,1}:(4_{\; 2}3_{\;1} 4) 
      \arrow{r}[pos=0.25,inner sep=2pt]{\;\;\;\; \alpha_i^{2,2}, \, i=2,1}  
    & \SSS^{2,2}:(4_{\; 2}2_{\;1} 4) 
      \arrow{r}[pos=0.25,inner sep=2pt]{\;\;\;\; \alpha_i^{2,3},\,  i=5,6,7}  
      &
    \SSS^{2,3}:(4_{\; 2}2_{\;1} 3) \\
\end{tikzcd}
\end{center}

\vspace{-0.7cm}
The nontrivial coefficients $\alpha_i^{k,n}$ necessary for each RDE step can be computed by integrating the MDB-spline functions of derivative spaces as in \eqref{eq:RKI_noint}. We remark that a similar result was proven in \cite{Shenetal2016} for a less general subclass of MD-splines.
  For instance, knowing the MDB-splines of space $D\SSS_1 \equiv D\SSS^{2,1}
  =\SSS^{1,1}$, which is defined on the same interval and breakpoint sequence as 
  $\SSS$, but has degrees $\bd =(3,2,3)$ and continuities $\KKK=(1,0)$, we can
  determine the coefficients $\alpha_i^{2,1}$, $i=4,5,6$, to pass from $\SSS_0 \equiv \SSS^{2,0}$
  to $\SSS_1 \equiv \SSS^{2,1}$. 
With reference to the first line of the rhomboid scheme, observe how going from
$\SSS^{0,0}$ to $\SSS^{0,1} $ it is necessary to calculate only one
coefficient, as one passes from degree $2$ to $1$. For the same reason it is
necessary to calculate only one coefficient  $\alpha_3^{0,3}$ for the RDE step from $\SSS^{0,2} $ to $\SSS^{0,3}$.
The RDE step from $\SSS^{0,1}$ to $\SSS^{0,2}$, instead, does not involve non-trivial RDE coefficients, as they are all equal to 0 or 1. 
\end{exmp}

\begin{rem}
The procedure can be modified in such a way to avoid any subtraction operation and therefore improve its numerical stability. 
In fact, the coefficients in the first row of the rhomboid scheme (in the example $\alpha_4^{0,1}$ and $\alpha_3^{0,3}$) are determined by
 \eqref{eq:KI_rev}, whose numerator can be computed by the middle line of
  \eqref{eq:dif_sum} (Algorithm \ref{alg:5}, lines 13 and 14).
 However, we can as well further differentiate these spaces, 
 in such a way that the first two rows of the scheme become:
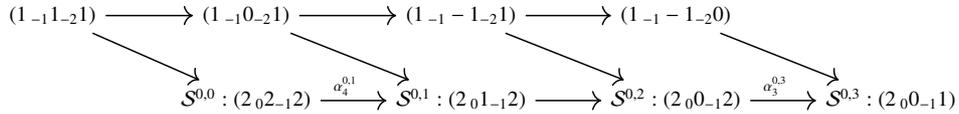
\begin{center}
\begin{tikzcd}[column sep={2.5em},scale cd=0.8]
    (1_{\; -1}1_{-2}1) 
     \arrow{r}[pos=0.25,inner sep=2pt]{}  
         \arrow{rd}[inner sep=1pt]{}  
     & (1_{\; -1}0_{-2}1) 
          \arrow{r}[pos=0.25,inner sep=2pt]{}  
         \arrow{rd}[inner sep=1pt]{}  
     & (1_{\; -1}-1_{-2}1) 
          \arrow{r}[pos=0.25,inner sep=2pt]{}  
         \arrow{rd}[inner sep=1pt]{}  
     & (1_{\; -1}-1_{-2}0)     
         \arrow{rd}[inner sep=1pt]{}  & \\ 
    & \SSS^{0,0}:(2_{\; 0}2_{-1}2) 
     \arrow{r}[pos=0.25,inner sep=2pt]{\;\;\;\; \alpha_{4}^{0,1} }  
  & \SSS^{0,1}:(2_{\; 0}1_{-1}2) 
 \arrow{r}[pos=0.25,inner sep=2pt]{}  
  & \SSS^{0,2}:(2_{\; 0}0_{-1}2) 
 \arrow{r}[pos=0.25,inner sep=2pt]{\;\;\;\; \alpha_{3}^{0,3} }  
  & \SSS^{0,3}:(2_{\; 0}0_{-1}1) \\
\end{tikzcd}
\end{center}

\vspace{-0.6cm}
In this extended version of the scheme also the coefficients $\alpha_4^{0,1}$ and  $\alpha_3^{0,3}$
can be determined through \eqref{eq:RKI_noint}, avoiding the aforementioned differences.
This variant of Algorithm \ref{alg:5} can be obtained by defining $r$ as
\[
 r \coloneqq \max\{d_i\}-1\,
 \]
 and summarizing lines from 12 to 18 of the algorithm by lines 16 and 17 only.
\end{rem}

The RDE-based algorithm is numerically stable for the same considerations made in the RKI case and all the numerical tests carried out have verified its excellent accuracy in the calculation of both the MDB-spline functions and the representation matrix.

\begin{rem}[Mixed RDE-RKI Algorithm]
It is also possible to design an algorithm that simultaneously performs RDE and RKI steps, like the one proposed in \cite{BC2020}.
In this case we shall choose the initial space $\SSS_0$, containing $\SSS$,  in such a way that 
$\SSS$ can be reached through a sequence of successive steps of RDE and RKI type.
We shall  hence break the target space $\SSS$
into sections, each of which will be generated from the corresponding section
  of $\SSS_0$ via RDE using Algorithm \ref{alg:5} or through RKI joins of the
  corresponding sections in $\SSS_0$ via Algorithm \ref{alg:2}.
At this point it is necessary to  proceed by first addressing all the sections requiring RDE, obtaining the representation matrices
of the corresponding MDB-spline bases, and then joining by RKI the resulting MD-spline spaces, 
starting from the one with the highest continuity up to the one with the least continuity.
\end{rem}

\section{Conclusions}\label{sec:conclusion}
We have presented an algorithm for the efficient evaluation of multi-degree B-splines, which, unlike previous approaches, is numerically stable. 
This has been emphasized via theoretical analysis of the involved operations, as well as by numerical experiments and comparisons with previous methods. 
From the point of view of numerical stability, the proposed method is at present the most effective tool for evaluating multi-degree splines.
Furthermore, similar ideas could be employed in the more general context of piecewise Chebyshevian splines of variable dimensions, which have been the subject of recent studies \cite{BCR2017,chebMD2019}.

\section*{Acknowledgements}
The authors gratefully acknowledge support from INdAM-GNCS Gruppo Nazionale per il Calcolo Scientifico.




\section*{References}

\ifelsevier
\bibliographystyle{model3-num-names} 
\bibliography{biblio} 
\fi



\ifspringer
\bibliographystyle{spmpsci} 
\bibliography{biblio} 
\fi

\end{document}